\documentclass[a4paper,12pt]{amsart}

\usepackage{amsfonts,amssymb,amscd,amsmath,latexsym,amsbsy,enumerate,stmaryrd,a4wide,verbatim,color}
\usepackage{chngcntr}
\numberwithin{figure}{section}
\numberwithin{table}{section}

\usepackage{tikz}

\theoremstyle{plain}
\newtheorem{theorem}{Theorem}[section]
\newtheorem{corollary}[theorem]{Corollary}

\newtheorem{cond}[theorem]{Condition}
\newtheorem{lemma}[theorem]{Lemma}
\newtheorem{prop}[theorem]{Proposition}

\theoremstyle{remark}
\newtheorem{remark}[theorem]{Remark}
\numberwithin{equation}{section}

\newcommand{\R}{\mathbb R}
\newcommand{\N}{\mathbb N}
\newcommand{\C}{\mathbb C}
\newcommand{\Z}{\mathbb Z}

\newcommand{\al}{\alpha}
\newcommand{\be}{\beta}
\newcommand{\ga}{\gamma}
\newcommand{\Ga}{\Gamma}
\newcommand{\de}{\delta}
\newcommand{\De}{\Delta}
\newcommand{\ep}{\varepsilon}
\newcommand{\si}{\sigma}

\newcommand{\la}{\lambda}
\newcommand{\Lam}{\Lambda}
\newcommand{\om}{\omega}
\newcommand{\Om}{\Omega}
\newcommand{\bt}{\mathbf{t}}
\newcommand{\bd}{\mathbf{d}}
\newcommand{\cA}{\mathcal{A}}
\newcommand{\cE}{\mathcal{E}}

\newcommand{\End}{\mathrm{End}}
\newcommand{\Hom}{\mathrm{Hom}}
\newcommand{\sph}{\mathrm{sph}}
\newcommand{\reg}{\mathrm{reg}}
\newcommand{\diag}{\mathrm{diag}}
\newcommand{\Tr}{\mathrm{Tr}}
\newcommand{\Id}{\mathrm{Id}}
\newcommand{\Ad}{\mathrm{Ad}}
\newcommand{\SU}{\mathrm{SU}}

\newcommand{\rFs}[5]{\,_{#1}F_{#2} \left( \genfrac{.}{.}{0pt}{}{#3}{#4}
	\ ;#5 \right)}
\newcommand{\Lg}{\mathfrak{g}}
\newcommand{\Lk}{\mathfrak{k}}
\newcommand{\Lm}{\mathfrak{m}}
\newcommand{\La}{\mathfrak{a}}
\newcommand{\Lh}{\mathfrak{h}}

\begin{document}
\title[$BC_2$ type multivariable matrix 
functions]{$BC_2$ type multivariable matrix 
functions \\ and matrix spherical functions}
\author{Erik Koelink, Jie Liu}
\address{ 
	IMAPP\\ Radboud Universiteit\\ P.O.~Box 9010\\ 6500 GL Nijmegen\\The Netherlands
}
\email{e.koelink@math.ru.nl}
\email{J.Liu@science.ru.nl}

\date{\today}

\begin{abstract} 
Matrix spherical functions associated to the compact symmetric pair
$(\mathrm{SU}(m+2), \mathrm{S}(\mathrm{U}(2)\times \mathrm{U}(m))$, $m\geq 2$, having reduced root system of type 
$\mathrm{BC}_2$, are studied. We consider an irreducible $K$-representation $(\pi,V)$ arising from the $\mathrm{U}(2)$-part of $K$,
and the induced representation $\mathrm{Ind}_K^G \pi$ splits  
multiplicity free. The corresponding spherical functions, 
i.e. $\Phi \colon G \to \mathrm{End}(V)$ 
satisfying $\Phi(k_1gk_2)=\pi(k_1)\Phi(g)\pi(k_2)$ for all $g\in G$, $k_1,k_2\in K$, are studied by studying certain 
leading terms which involve hypergeometric
functions. 
This is done explicitly using the 
action of the 
radial part of the Casimir operator on these functions and their leading terms. To suitably grouped matrix spherical functions we associate two-variable matrix orthogonal polynomials giving a matrix analogue of Koornwinder's 1970s two-variable orthogonal polynomials, which are Heckman-Opdam polynomials for $\mathrm{BC}_2$. 
In particular, we find explicit orthogonality relations 
and the matrix polynomials being eigenfunctions to an
explicit second order 
matrix partial differential operator. 
The scalar part of the matrix weight is less general than 
Koornwinder's weight. 
\end{abstract}

\maketitle

\section{Introduction}

Spherical functions on compact symmetric spaces and orthogonal polynomials are closely related ever since the work
of \'E. Cartan, see e.g. \cite{GangV}, \cite{Helg-1962}, \cite{Warn-2}.   
The notion of a spherical function taking values in 
a matrix algebra goes back to the initial introduction of the 
notion of spherical functions, see e.g. 
\cite[Introduction]{GangV} and references given there. 
In case of a matrix spherical function for a compact symmetric space of rank one, there is a connection
to matrix orthogonal polynomials. One of the first papers 
in this direction is Koornwinder's paper \cite{Koor-SIAM1985} 
introducing vector valued polynomials, which can be written 
as matrix orthogonality, see also \cite{KoelvPR-IMRN}, \cite{KoelvPR-PRIMS}. The vector polynomials are evaluated in an explicit way in terms 
of the representations of $\SU(2)$, see \cite[Prop.~3.2]{Koor-SIAM1985}.
Another seminal paper 
making this connection to matrix polynomials explicit is the paper 
\cite{GrunPT} by Gr\"unbaum, Pacharoni and Tirao, where they
study the rank one symmetric space  $(\SU(3), \mathrm{U}(2))$. The approach of \cite{GrunPT} relies on the study of the invariant differential operators on the corresponding homogeneous space. 
Since then several other approaches have been explored, and 
many other rank one cases have been studied in detail. 
For this paper the approach of \cite{KoelvPR-IMRN}, 
\cite{KoelvPR-PRIMS}, \cite{KoelvPR-JFA} is the most relevant,
see Chapter 13 by Gr\"unbaum, Pacharoni, Tirao in 
\cite{Isma-AB} for other approaches and references. 

The scalar spherical functions on symmetric spaces have been vastly generalised in the work of Heckman and Opdam, see  Heckman's lecture notes in \cite{HeckS}, or Chapter 8 by Heckman and Opdam
in \cite{KoorS}.
The root multiplicities, i.e. dimensions of root spaces, arising from the symmetric spaces are considered as more general continuous parameters, and the 
second order partial differential operator extending 
the radial part of the Casimir operator for the symmetric space plays an important role. A first important step was taken by Koornwinder in the 1970s, who studied several 
sets of orthogonal polynomials in two variables generalising the spherical functions arising 
for type $\mathrm{A}_2$ and $\mathrm{BC}_2$. 
As a first step for a matrix generalisation, the matrix spherical functions and corresponding matrix orthogonal 
polynomials need to be considered. For the type 
$\mathrm{A}_n$ this is done in \cite{KoelvPR-JFA}, and 
the purpose of this paper is to study 
the matrix spherical functions and corresponding matrix orthogonal polynomials for type $\mathrm{BC}_2$. A possible next step is then to generalise to more general 
parameters, one possibility is using shift operators for 
the classical case of $\mathrm{BC}_2$, see Opdam \cite[\S 2]{Opda}, and employ the same shift operator to the 
matrix case as well. This has been done successfully in the rank one case to go from matrix Chebyshev polynomials to 
matrix Gegenbauer polynomials, see  \cite{KoeldlRR}. 
We expect that the interpretation can lead to more properties of the 
corresponding matrix orthogonal polynomials studied in this paper. Moreover, the relation to possible applications 
in mathematical physics needs to be investigated, see e.g.
\cite{StokR} for more information and references given there.

In this paper we study the matrix spherical 
functions for the compact symmetric pair 
$(G,K)=(\mathrm{SU}(m+2), \mathrm{S}(\mathrm{U}(2)\times \mathrm{U}(m))$, and we study matrix spherical 
functions and corresponding matrix orthogonal 
polynomials as described in Section \ref{ssec:generalsetup} for the case of an irreducible representation of
$K$ arising from the $\mathrm{U}(2)$-component in $K$. 
The results of Subsection \ref{ssec:generalsetup} follows  \cite[Part I]{KoelvPR-JFA}, but there are slight variations
on this approach, see \cite[\S 9]{PezzvP}. 
Actually, we use the classification of \cite{PezzvP}
in order to find the right $K$-representations satisfying
the multiplicity free Condition \ref{cond:mu-multfree},
but \cite{PezzvP} gives more possibilities, i.e. also 
involving other $K$-representations. In this 
paper we restrict to the $K$-representations 
arising from the $\mathrm{U}(2)$-block in $K$ with a slight
assumption on this representation. 
In this paper we show that instead of studying the 
more complicated matrix spherical functions we can study more simple leading terms of the matrix spherical 
functions. The leading terms turn out to 
be homogeneous polynomials, and homogeneity considerations allow us to prove some 
results, e.g. on the indecomposability of the 
corresponding matrix weight and the explicit derivation 
of the second order matrix partial differential equation. 
Initially, we study the leading terms for the matrix spherical functions for labels in 
$B(\mu)$ as defined in Condition \ref{cond:BmustructurePGmu}, see  
Theorem \ref{thm:Qmunuisasmatrixspherfunction}. 
Note that to study the matrix spherical functions explicitly 
we need explicit control over the $K$-intertwiner embedding a specific $K$-representation into a larger irreducible $G$-representation. This is in general hard to do explicitly, but this 
approach is used successfully in \cite{KoelvPR-JFA} for the symmetric
pair correspondig to the group case for type $\mathrm{A}$. 
In this paper we take an alternative approach and we construct the embedding of the specific $K$-representation into a larger tensor product $G$-representation
containing the required irreducible $G$-representation 
as a constituent in the decomposition. 
Then we have to show that the embedding indeed `sees' the 
appropriate irreducible $G$-representation. Of course, there
are many ways to do this, and in this paper we motivate the choice we make as follows.  Firstly it leads to a leading term whose components are homogeneous polynomials and secondly, the radial part of the Casimir operator on the leading term has a simple expression, 
see Lemma \ref{lem:acrionradpartonQs}. The approach taken is motivated by the preprint \cite{Prui-pp} by van Pruijssen.

In order to make the connection between the leading terms and the matrix spherical functions explicit we need the 
action of the radial part of the Casimir operator as an operator acting
on matrix valued functions on $A$. For completeness this action is derived in Appendix \ref{sec:AppRadial}. 
For the matrix spherical functions corresponding to
elements from $B(\mu)$ as in Condition \ref{cond:BmustructurePGmu} we find an explicit
expression in this way in terms of the leading terms, 
see Proposition \ref{prop:explicitQsinPhis}. 
Then in Section \ref{sec:Leadingcoeffgencase}
we obtain the leading terms for the general case, and we show that the radial part of the Casimir operator 
acts in a lower triangular way with respect to the 
partial ordering. This is analogous to the case for the (scalar) Heckman-Opdam polynomials, see \cite[\S 1.3]{HeckS}. 
The main result is Theorem \ref{thm:MVOPBC2} in which we explicitly give the matrix orthogonality for the 
corresponding family of two-variable orthogonal polynomials 
with an explicit matrix weight on a region bounded by two 
straight lines and a parabola, see Figure \ref{fig:Integrationregion}. Theorem \ref{thm:MVOPBC2}
also states that these matrix polynomials are eigenfunctions
of a second order matrix partial differential operator. 

We now describe the content of the paper in brief. 
In the remaining part of the Introduction we briefly recall in 
Subsection \ref{ssec:generalsetup} 
the set-up to go from matrix spherical function to 
matrix orthogonal polynomials, where the number of 
variables is equal to the rank of the compact symmetric 
space. This follows \cite[Part I]{KoelvPR-JFA}. 
In Section \ref{sec:structuremultfree} we
briefly describe the structure theory and notation for 
the compact symmetric pair 
$(G,K)=(\mathrm{SU}(m+2), \mathrm{S}(\mathrm{U}(2)\times \mathrm{U}(m))$, $m>2$, and we show that the conditions in 
Subsection \ref{ssec:generalsetup} are satisfied in this case. 
In Section \ref{sec:specialcases} we 
develop the building blocks for the leading terms.
These are essentially the leading terms in the case of 
the $K$-representation in Subsection \ref{ssec:generalsetup}
corresponding to the trivial representation and to 
the natural representation of the $\mathrm{U}(2)$-block 
in $K$. Building on this we study the leading 
term for the matrix spherical functions 
corresponding to $B(\mu)$ 
as in Condition \ref{cond:BmustructurePGmu}. The leading terms can be fully described in terms of single variable 
Krawtchouk polynomials, and hence as single variable hypergeometric functions. 
Next in Section \ref{sec:leadingcoefQnu} we use the radial part of the Casimir operator to give an 
explicit expression of the matrix spherical functions corresponding to $B(\mu)$ in terms of the leading terms. 
In Section \ref{sec:MatrixWeight} we describe the two-variable matrix weight, and we show that the weight is 
indecomposable and that its determinant is non vanishing on the interior of the integration region. 
In Section \ref{sec:MVOP} we describe the two variable matrix orthogonal polynomials, and we describe the corresponding eigenvalue equation involving a 
second order matrix partial differential operator. 
We have choosen the coordinates in Theorem \ref{thm:MVOPBC2} so that it matches the notation 
of Koornwinder \cite{Koor-IndagM1}, \cite{Koor-Wisconsin},
see also \cite{Spri}. Theorem \ref{thm:MVOPBC2} 
generalises the results of \cite{Koor-IndagM1}, 
\cite{Koor-Wisconsin}, \cite{Spri} to the matrix case, but 
the scalar part of the weight measure in 
\cite{Koor-IndagM1}, 
\cite{Koor-Wisconsin}, \cite{Spri} is more general 
than the one in Theorem \ref{thm:MVOPBC2}.
Theorem \ref{thm:MVOPBC2} also contains the 
case \cite[Ch.~5]{HeckS} for this particular symmetric pair (corresponding to the case $a=0$ in the notation 
of Section \ref {sec:MVOP}). In 
Section \ref{sec:Leadingcoeffgencase} we then 
derive the leading term for the general matrix 
spherical functions, and we show that the radial 
part of the Casimir operator acts in a lower triangular
fashion on such a leading term.
Finally, in Section \ref{sec:bnegative} we discuss
briefly the remaining cases of the 
$K$-representations of this type. 

In the course of several proofs we have to manipulate 
several expressions involving functions in two variables. We have used computer algebra, in particular Maple and Maxima, to check these computations.

\medskip
\textbf{Acknowledgement.} We thank Maarten van Pruijssen for useful discussions and feedback and we thank the referee for useful comments. 
We would like to thank the China Scholarship Council for funding support.

\subsection{General set-up}\label{ssec:generalsetup}
In this subsection we 
recall notations and the necessary results. We follow
\cite[Part 1]{KoelvPR-JFA}, but see also 
\cite[\S 11]{PezzvP}, \cite{Prui-IMRN}.  
We consider a compact
symmetric pair $(G,K)$ and for its structure theory and 
results we refer to \cite{Helg-1962}. 
For the explicit case 
$(G,K)=(\mathrm{SU}(m+2), \mathrm{S}(\mathrm{U}(2)\times \mathrm{U}(m))$ the structure theory is explicitly given in  Section \ref{sec:structuremultfree}. 
We label the representations
of $G$, respectively $K$, by the highest weights 
$P^+_G$, respectively $P^+_K$, and such a representation 
is denoted by $(\pi_\la^G, V^G_\la)$, $\la\in P^+_G$, and similarly for $K$. We now fix $\mu\in P^+_K$.

In order to apply the general approach of \cite{KoelvPR-JFA} we need to establish three conditions. 

\begin{cond}\label{cond:mu-multfree}
$\mathrm{Ind}_K^G \pi^K_\mu$ splits multiplicity free. 
\end{cond}

By Frobenius reciprocity this is equivalent to 
$[\pi^G_\la\vert_K \colon \pi^K_\mu]\leq 1$ for all $\la\in P^+_G$, and we put 
\begin{equation}\label{eq:defP+Gmu}
P^+_G(\mu) =\{ \la \in P^+_G \mid [\pi^G_\la\vert_K \colon \pi^K_\mu] = 1\}.
\end{equation}
So, if Condition \ref{cond:mu-multfree} holds, we have 
\begin{equation*}
\mathrm{Ind}_K^G \pi^K_\mu =\bigoplus_{\la \in P^+_G(\mu)} V^G_\la. 
\end{equation*}
For $\la\in P^+_G(\mu)$ we define the corresponding matrix spherical 
functions
\begin{equation}\label{eq:defMSphericalFunction}
\Phi^\mu_\la \colon G \to \End(V^K_\mu), \qquad
\Phi^\mu_\la(g) = p\circ \pi^G_\la(g) \circ j,
\end{equation}
where $j\in \Hom_K(V^K_\mu, V^G_\la)$ is the unitary intertwiner and $p=j^\ast$ is the corresponding $K$-equivariant orthogonal projection. Then
\eqref{eq:defMSphericalFunction} is independent of the choice of $j$ and  we have 
\begin{equation}\label{eq:invariancepropMSF}
\Phi^\mu_\la(k_1gk_2) = \pi^K_\mu(k_1)\, \Phi^\mu_\la(g)\, \pi^K_\mu(k_2), \qquad \forall\, k_1,k_2\in K, \ \forall\, g\in G. 
\end{equation}
The space of regular functions $\Phi\colon G\to \End(V^K_\mu)$
satisfying the left and right $K$ transformation behaviour as in  
\eqref{eq:invariancepropMSF} is denoted by $E^\mu$. Using the 
Peter-Weyl decomposition we see that $\{\Phi^\mu_\la\mid \la\in P^+_G(\mu)\}$ forms a linear basis for $E^\mu$. Then 
$E^0$ is the space of scalar continuous bi-$K$-invariant functions, 
and $E^\mu$ is a $E^0$-module. Moreover, Schur orthogonality gives
\begin{equation}\label{eq:orthorelMatrixSpherF}
\int_G \Tr \bigl( \Phi^\mu_\la(g) (\Phi^\mu_{\la'}(g))^\ast\bigr) 
\, dg = \de_{\la,\la'} \frac{(\dim V^K_\mu)^2}{\dim V^G_\la}, 
\qquad \la,\la'\in P^+_G(\mu).
\end{equation}
Note that the integrand is a bi-$K$-invariant function, so contained in $E^0$.

Let $A$ be the abelian subgroup and $M=Z_K(A)$ as in \cite[\S 2]{KoelvPR-JFA}. By the Cartan decomposition, $G=KAK$, and by 
\eqref{eq:invariancepropMSF} it suffices to consider 
\begin{equation}
\Phi^\mu_\la\vert_A \colon A \to \End_M(V^K_\mu),
\end{equation}
since $\pi^K_\mu(m) \Phi^\mu_\la(a) = 
\Phi^\mu_\la(ma) = \Phi^\mu_\la(am) = \Phi^\mu_\la(a)\pi^K_\mu(m)$. So we need to know the decomposition 
\begin{equation}\label{eq:decompVKmurestrtoM}
V^K_\mu\vert_M \cong \bigoplus_{i=1}^N V^M_{\si_i}
\end{equation}
where $\si_i\in P^+_M$ are the corresponding highest weights for $M$.  The decomposition \eqref{eq:decompVKmurestrtoM} is again a multiplicity free decomposition, see \cite{Prui-IMRN} and also \cite{Deit}, \cite{Koba}. 

Note that if the representation $\pi_\mu^K$ induces multiplicity free,
then also its dual (or contragredient) representation 
$(\pi_\mu^K)^\ast = \pi_{\mu^\ast}^K$ induces multiplicity free, where 
$\mu^\ast$ corresponds to the highest weight of the 
dual representation. Then $P^+_G(\mu^\ast)$ consists of those 
$G$-representations for which the dual is in $P^+_G(\mu)$, i.e.
$P^+_G(\mu^\ast) = \{ \la^\ast \mid \la \in P^+_G(\mu)\}$, where 
$\la^\ast$ corresponds to the highest weight of the dual of 
the $G$-representation with highest weight $\la$. Then we obtain
\begin{equation}\label{eq:dualofPhimula}
\bigl( \Phi^{\mu^\ast}_{\la^\ast}(a) v^\ast\bigr)(v) = 
v^\ast\bigl( \Phi^{\mu}_{\la}(a^{-1}) v\bigr), 
\qquad a\in A, \quad v\in V_\mu^K, \quad v^\ast \in \Hom(V_\mu^K,\C)=V^K_{\mu^\ast}.  
\end{equation}
Note that if Condition \ref{cond:mu-multfree} holds, then it also
holds for the dual $\mu^\ast\in P^+_K$. Moreover, 
taking duals gives an involution on the spherical weights $P^+_G(0)$.

\begin{cond}\label{cond:BmustructurePGmu}
There exists a set of weights $B(\mu) \subset P^+_G$, so that for 
$\la\in P^+_G(\mu)$ there exist unique elements $\nu\in B(\mu)$ and 
$\la_\sph \in P^+_G(0)$ with $\la = \nu+\la_\sph$. 
The restriction map of the torus of $G^\C$ to the torus of $M^\C$ gives a
bijection $B(\mu) \overset{\cong}{\to} 
\{ \si \in P^+_M \mid [V^K_\mu\vert_M\colon V^M_\si]=1\}$.
\end{cond}

Assuming Condition \ref{cond:BmustructurePGmu} is satisfied for $\mu\in P^+_K$, then Condition \ref{cond:BmustructurePGmu} is also satisfied
for the dual $K$-representation with highest weight $\mu^\ast$. 

Taking $\mu=0$, $P^+_G(0)$ corresponds to the spherical weights, and $P^+_G(0) =\bigoplus_{i=1}^n \N \la_i$,  where 
$\la_1,\cdots, \la_n$ are the generators for the spherical weights and $n$ is the rank of the compact symmetric space $(G,K)$. We let $\phi_i = \Phi^0_{\la_i} \colon G \to \C$, which generate the algebra of bi-$K$-invariant polynomials 
on $G$. For $\la=\sum_{i=1}^n d_i \la_i\in P^+_G(0)$ we put $|\la|=\sum_{i=1}^nd_i$. 
We use the notation $P_G(\la)$ for all the weights occurring in the 
$G$-representation $\pi^G_\la$ of highest weight $\la\in P^+_G$, and similarly for other groups. 

\begin{cond}\label{cond:degreeconstraint}
For all weights $\nu \in B(\mu)$, for  all generators $\la_i$ of 
the spherical weights $P_G^+(0)$ and for all weights $\eta \in  P_G(\la_i)$
such that $\nu+\eta\in  P^+_G(\mu)$, we have by Condition 
\ref{cond:BmustructurePGmu} a unique $\nu'\in B(\mu)$ such that
$\nu+\eta = \nu'+\la$ with $\la\in P^+_G(0)$.  Then $|\la|\leq 1$.
\end{cond}

Note that if Condition \ref{cond:degreeconstraint} holds for 
$\mu$, then it also holds for the dual $\mu^\ast\in P^+_K$. 

Assuming Conditions \ref{cond:mu-multfree}, \ref{cond:BmustructurePGmu} and \ref{cond:degreeconstraint} one
can show that for $\la_\sph=\sum_{r=1}^n d_r \la_r\in P^+_G(0)$,
$\bd = (d_1,\cdots, d_n)\in \N^n$, there exist unique $n$-variable 
polynomials $p^\mu_{\nu_i,v_r;\bd}$ of total degree $|\bd|=|\la_\sph|$ so that 
for $\la=\nu_i+\la_\sph\in P^+_G(\mu)$ and $a\in A$ 
\begin{equation}\label{eq:expPhumulaintoPhibottom}
\Phi^\mu_\la(a) = \Phi^\mu_{\nu_i+\la_\sph}(a)  
= \sum_{r=1}^N p^\mu_{\nu_i,\nu_r;\bd}(\phi_1(a),\cdots, \phi_n(a))
\, \Phi^\mu_{\nu_r}(a)
\end{equation}
using a slighlty different labeling from \cite{KoelvPR-JFA}. Using 
this expansion in the orthogonality relations \eqref{eq:orthorelMatrixSpherF} and reducing the integral  for bi-$K$-invariant 
functions to an integral over $A$, see 
\cite[Prop.~X.1.19]{Helg-1962}, we find the matrix orthogonality 
relations 
\begin{equation}\label{eq:MatrixOPorthogonality}
\int_A P_{\bd}(\phi(a)) W(\phi(a)) \bigl( P_{\bd'}(\phi(a))\bigr)^\ast \, |\de(a)|\, da = c \de_{\bd, \bd'} H_{\bd}, \\
\end{equation}
where $H_\bd$ is given by
$(H_\bd)_{i,j} = \de_{i,j} (\dim V^K_\mu)^2/\dim V^G_{\nu_i + \la_\sph}$, $\phi(a)=(\phi_1(a),\cdots, \phi_n(a))$, 
\begin{gather*}
P_{\bd}(\phi(a)) = \bigl( p_{\nu_i, \nu_j; \bd}(\phi_1(a),\cdots,
\phi_n(a))\bigr)_{i,j=1}^N, 
\quad W(\phi(a)) = \bigl( \Tr\bigl( \Phi^\mu_{\nu_i}(a) (\Phi^\mu_{\nu_j}(a))^\ast\bigr)\bigr)_{i,j=1}^N
\end{gather*}
and $c>0$ is determined by $c=\int_A |\de(a)|\, da$ and $\de$ is 
given in \cite[Prop.~X.1.19]{Helg-1962}, where it is denoted as 
$D_\ast$.

\section{Structure theory and multiplicity free triples}
\label{sec:structuremultfree}

In this section we specialise to the compact symmetric pair 
$(G,K)=(\mathrm{SU}(m+2), \mathrm{S}(\mathrm{U}(2)\times \mathrm{U}(m))$, $m>2$,  for which we study the matrix spherical functions and the related orthogonal polynomials in detail.
First we describe the structure theory, see e.g. 
\cite{Helg-1962}, needed in order to associate the corresponding 
orthogonal polynomials in Subsection \ref{ssec:structurethy}.
In the remaining part we show that for the 
explicit $K$-representations the conditions 
of \cite[Part I]{KoelvPR-JFA} are satisfied in this case. 

\subsection{Structure theory}\label{ssec:structurethy}
From now on we take $G=\SU(m+2)$, $m>2$, $K=\mathrm{S}(\mathrm{U}(2)\times \mathrm{U}(m))$
embedded block-diagonally. We view $\mathrm{U}(2)\subset K$ as subgroup as the upper left hand $2\times 2$-block of $K$. The abelian subgroup is 
$A = \{ a_\bt = a_{(t_1,t_2)} \mid t_1,t_2\in \R\}$, with 
\begin{equation}\label{eq:defat}
a_\bt =  a_{(t_1,t_2)} =
\begin{pmatrix}
\cos t_1 & 0 & 0 & \cdots &0 & 0 & i\sin t_1 \\
0 & \cos t_2 & 0 & \cdots &0 & i\sin t_2 & 0 \\
0 & 0 & 1 & \cdots &0 & 0  & 0 \\
\vdots & \vdots &  & \ddots& & \vdots & \vdots \\
0 & 0 & 0 & \cdots & 1 & 0& 0 \\
0 & i\sin t_2 & 0 & \cdots & 0 & \cos t_2 & 0 \\
i\sin t_1 & 0 & 0  & \cdots & 0 & 0 & \cos t_1 \\
\end{pmatrix}
\end{equation}
where the middle block is the $(m-2)\times (m-2)$-identity matrix. Then 
$M=Z_K(A)$ is given by matrices $m$ which are block diagonal of size 
$2\times 2$, $(m-2)\times (m-2)$, $2\times 2$ of the form 
\begin{equation}\label{eq:defminM}
m = \begin{pmatrix} D_1 & 0  & 0 \\ 0 & M_1 & 0 \\ 0 & 0 & D_2\end{pmatrix}, 
\quad D_1 = \begin{pmatrix} e^{is_1}& 0 \\ 0 & e^{is_2} \end{pmatrix}, 
\  M_1\in \mathrm{U}(m-2), \  
D_2 = \begin{pmatrix} e^{is_2}& 0 \\ 0 & e^{is_1} \end{pmatrix}
\end{equation}
with $\det(m)=1$. 

As the torus of $G^\C$ we take the diagonal elements, and 
we take this also as the torus of $K^\C$. Explicitly,  
\begin{equation}
T_{G^\C} = T_{K^\C} = \{ \diag(t_1,\cdots, t_{m+2}) \mid t_k\in \C, \prod_{i=1}^{m+2} t_i =1\}.
\end{equation}
We take the torus of $M^\C$ as contained in the torus of 
$G^\C$ and $K^\C$, 
\begin{equation}
T_{M^\C} = \{ \diag(t_1,\cdots, t_{m+2}) \mid  t_{m+1}=t_2, t_{m+2}=t_1, \prod_{i=1}^{m+2} t_i =1\} \subset T_{G^\C} = T_{K^\C}.
\end{equation}

By $\Lg$, $\Lk$, $\Lm$ and $\La$ we denote the corresponding 
complexified Lie algebras of $G$, $K$, $M$ and $A$. Then
the root system $\De$ of $\Lg$ is of type $\mathrm{A}_{m+1}$, and we
denote the standard simple roots $\al_i$, $1\leq i \leq m+1$. We 
put $E_i= E_{i,i+1}$, $F_i=E_{i+1,i}$, $H_i= E_{i+1,i+1}-E_{i,i}$, 
where $E_{i,j}$ is the matrix with all zeroes except the $(i,j)$-th entry. 
The roots and 
positive roots are denoted by $Q_G = \bigoplus_{i=1}^{m+1} \Z \al_i$ and $Q^+_G = \bigoplus_{i=1}^{m+1} \N \al_i$. The partial order
$\si \preccurlyeq \eta$ is $\eta-\si\in Q^+_G$. 

With this choice of positive roots, we define the fundamental weights for $G$, $K$ and $M$ by 
\begin{equation}
\begin{split}
\om_i \colon T_{G^\C} = T_{K^\C} \to \C, \qquad \om_i(\diag(t_1,\cdots, t_{m+2}))= \prod_{j=1}^i t_j, \quad 1\leq i<m+2, \\
\eta_i \colon T_{M^\C} \to \C, \qquad \eta_i(\diag(t_1,t_2, \cdots, t_m, t_2, t_1))= \prod_{j=1}^i t_j, \quad 1\leq i<m.
\end{split}
\end{equation}
Note that $\eta_1$, $\eta_2$ are characters of $M$. 
Then we find 
\begin{equation}\label{eq:restromitoTM}
\om_i\vert_{T_{M^\C}} = \eta_i \quad (1\leq i<m), \qquad 
\om_m\vert_{T_{M^\C}} = -\eta_2, \qquad 
\om_{m+1}\vert_{T_{M^\C}} = -\eta_1.
\end{equation}
Then we have 
\begin{equation}
P^+_K =\{ \sum_{i=1}^{m+1} a_i \om_i\mid a_2\in \Z, a_i\in \N, i\not=2\},
\qquad P^+_G = \bigoplus_{i=1}^{m+1} \N \om_i. 
\end{equation}
Considering $\mathrm{U}(2)\subset K$, we see that the $\mathrm{U}(2)$ representations correspond to the elements of $P^+_K$ with $a_j=0$ for 
$j\geq 3$. 

The reduced root system is of type $\mathrm{BC}_2$, and the 
corresponding reduced Weyl group is generated by $s_1$ and 
$s_2$, and we put $n_1,n_2\in N_K(A)$ by 
\begin{equation}\label{eq:defn1n2forW}
n_1 =
\begin{pmatrix}
J_2 & 0 & 0 \\ 0 & I_{m-2} & 0 \\ 0 & 0 & J_2 
\end{pmatrix},
\qquad 
n_2 = \diag(1,-i, \underbrace{1,\cdots,1}_{m-2}, i,1)
\end{equation}
using the notation of Appendix \ref{ssec:appAstructurethy} for 
the flip $J_2$. 
Then $n_1a_\bt n_1^{-1}= a_{s_1\bt}$, $n_2a_\bt n_2^{-1}=
a_{s_2\bt}$ with $s_1\bt=(t_2,t_1)$ and $s_2 \bt =(t_1,-t_2)$. 

\subsection{Multiplicity free triples}\label{ssec:multfreetriples}

The triple $(G,K,\mu)$, $\mu\in P^+_K$, is a multiplicity free triple if
Condition \ref{cond:mu-multfree} is satisfied. 
Since $(G,K)$ is a symmetric pair, the triple 
$(G,K,0)$, where $\mu=0$ corresponds to the trivial $K$-representation, is a multiplicity free triple. Then we have the spherical weights 
\begin{equation}\label{eq:defla1la2}
P^+_G(0) = \N \la_1 \oplus \N\la_2, \qquad 
\la_1 = \om_1+\om_{m+1}, \quad \la_2 = \om_2+\om_m,
\end{equation}
see Kr\"amer \cite[Tabelle~1]{Kram}. More generally, the multiplicity free triples and the set $P^+_G(\mu)$ for a multiplicity free triple
$(G,K,\mu)$ are determined by 
Pezzini and van Pruijssen \cite{PezzvP}. We focus on representations
of $K$ that correspond to  representations of $\mathrm{U}(2)\subset K$, i.e. we assume $\mu=a\om_1+b\om_2$, $a\in \N$, $b\in \Z$. 

\begin{prop}\label{prop:P+Gmu} The triple $(G,K, \mu)$, with 
$\mu=a\om_1+b\om_2$, $a\in \N$, $b\in \Z$, is multiplicity free. 
Moreover, $P^+_G(\mu)=B(\mu) + P^+_G(0)$. 
In case $b\in \N$ we have
\begin{equation*}
B(\mu)  = 
\{ \nu_i=\nu_i(\mu) = (a-i) \om_1 + (i+b)\om_2+ i\om_{m+1} 
\mid 0\leq i \leq a \}. 
\end{equation*}
In case $b\leq -a$, we have 
\begin{equation*}
B(\mu)  = 
\{ \nu_i=\nu_i(\mu) =
(a-i) \om_1 + (-i-b)\om_m + i\om_{m+1}  
\mid 0\leq i \leq a \}. 
\end{equation*}
and in case $-a < b < 0$ we have
\begin{multline*}
B(\mu)  = 
\{ \nu_i=\nu_i(\mu) = 
(a-i) \om_1 + (-i-b)\om_m  + i\om_{m+1} 
\mid 0\leq i < -b \} \\ \cup 
\{ \nu_i=\nu_i(\mu) = 
(a-i) \om_1 + (b+i)\om_2  + i\om_{m+1} 
\mid -b\leq i \leq a\} 
\end{multline*}
\end{prop}

\begin{remark}\label{rmk:dualityforGKandBmu}
Recall that the $G$-representation of highest weight $\om_i$ can 
be realised in the exterior power $\Lam^iV$, where $V=\C^{m+2}$ 
is the natural $G$-representation. It follows that $\om_i^\ast = \om_{m+2-i}$, and this determines $\la^\ast$. 
For the spherical weights, see \eqref{eq:defla1la2},
$\la_1^\ast=\la_1$ and $\la_2^\ast=\la_2$. The dual of the 
$K$-representation of highest weight $\mu=a\om_1+b\om_2$ is the $K$-representation
of highest weight $\mu^\ast=a\om_1-(a+b)\om_2$. Indeed, 
the map $v\mapsto \langle v, v_\mu\rangle$, with $v_\mu$ the highest weight vector of $V^K_\mu$, is the lowest weight vector of
$(V^K_\mu)^\ast$ of weight $-a\om_1-b\om_2$. 
Then $B(\mu^\ast)= \bigl(B(\mu)\bigr)^\ast$, and more 
precisely, in the notation of Proposition \ref{prop:P+Gmu},
$\nu_i(\mu^\ast) = \bigl(\nu_{a-i}(\mu)\bigr)^\ast$. 
\end{remark}

\begin{proof}
The proof is a verification using the results and notation 
of \cite{PezzvP}, in particular \cite[Table B.2.1]{PezzvP}. As noted after 
\cite[Def.~9.1]{PezzvP} we have $[\pi^G_\la\vert_K\colon \pi^K_\mu]=1$
if and only if $(\la,-\mu)$ is an element of the so-called 
extended weight monoid $\tilde{\Gamma}(G/P)$, $P\subset K^\C$ corresponding parabolic subgroup, corresponding to $G/K$. 
Now we use \cite[Table B.2.1]{PezzvP} to see that the elements 
of $\tilde{\Gamma}(G/P)$ are nonnegative integral linear combinations of 
\begin{gather*}
(\om_1+\om_{m+1},0), \quad 
(\om_1,-\om_1), \quad 
(\om_2,-\om_2), \quad 
(\om_{m},\om_2), \quad 
(\om_{m+1},\om_2-\om_1). 
\end{gather*}
We then see that $(\la,0)\in \tilde{\Gamma}(G/K)$ if and only 
if $\la\in P^+_G(0)$, i.e. $\la$ is a spherical weight. It is now
a straightforward calculation to determine the 
$\la\in P^+_G$ satisfying $(\la, a\om_1+b\om_2)
\in \tilde{\Ga}(G/K)$. 
\end{proof}

Note that the representation of $K$ with highest weight
$\mu=a\om_1+b\om_2$ has dimension $a+1$. Denoting the highest weight vector by $v_\mu$ we see that $V^K_\mu$ has an
orthogonal basis $\{ v_k= F_1^k\cdot v_\mu\mid 0\leq k\leq a\}$ by 
considering the representation as a $\mathrm{U}(2)$-representation.
It follows that, taking $m\in M$ as in \eqref{eq:defminM}, we 
have $\pi_\mu^K(m)v_k = e^{i(a+b-k)s_1}e^{i(b+k)s_2} v_k$, so that 
this corresponds to the $M$-weight $(a-2k)\eta_1+(b+k)\eta_2$. So 
\begin{equation}\label{eq:decompVKmuinMreps}
V^K_{\mu}\vert_M = \bigoplus_{k=0}^a  
V^M_{\si_k}, \qquad \si_k= \si_k(\mu) = (a-2k)\eta_1+(b+k)\eta_2
\end{equation}
splits multiplicity free into $1$-dimensional $M$-representations. 
Since the $M$-representations are $1$-dimensional, we find 
$\si_k^\ast=-\si_k$ and $\si_k(\mu^\ast) =  
\si_{a-k}(\mu)^\ast$. 

In any of the cases of Proposition \ref{prop:P+Gmu} 
we have $\nu_i(\mu)\vert_{T_{M^\C}} =  \si_i(\mu)$ using 
\eqref{eq:restromitoTM}. This leads to Corollary 
\ref{cor:prop:P+Gmu}. 

\begin{corollary}\label{cor:prop:P+Gmu}
For $\mu=a\om_1+b\om_2 \in P^+_K$, $a\in \N$, $b\in \Z$, 
Conditions \ref{cond:mu-multfree} and \ref{cond:BmustructurePGmu} are satisfied. 
\end{corollary}

\begin{proof} The statement can obtained by analysing
more carefully the extended weight monoid of \cite{PezzvP} used in the 
proof of Proposition \ref{prop:P+Gmu}, but it can also be done directly having the $B(\mu)$ at hand. Assume $\la\in P^+_G(\mu)$ can be written
as $\nu_i+\la_\sph=\nu_j+\la_\sph'$, with $\la_\sph=n_1\la_1+n_2\la_2$, 
$\la_\sph'=m_1\la_1+m_2\la_2$. Assume first $\mu=a\om_1+b\om_2$ with $b\in\N$, then we have, using Proposition \ref{prop:P+Gmu} 
and \eqref{eq:defla1la2},  
\begin{gather*}
0=\nu_i+\la_\sph-\nu_j-\la_\sph'= \\
(j-i+n_1-m_1)\om_1 + (i-j+n_2-m_2)\om_2 + (n_2-m_2)\om_m 
+ (i-j+n_1-m_1)\om_{m+1}.
\end{gather*}
This gives $n_2=m_2$, $i=j$ and $n_1=m_1$, and uniqueness follows. 
The case $b\leq -a$ follows by duality, and the case $-a<b<0$ can be proved similarly taking into account the different cases in 
Proposition \ref{prop:P+Gmu}.

The fact that the restriction map gives an isomorphism of 
$B(\mu)$ and the set of irreducible $M$-modules in $V^K_\mu\vert_M$
follows from \eqref{eq:decompVKmuinMreps} and 
\eqref{eq:restromitoTM}. 
\end{proof}

\subsection{Condition \ref{cond:degreeconstraint}}\label{ssec:condition3}
In order to able to apply the general theory as described in 
Subection \ref{ssec:generalsetup} we need to check
Condition \ref{cond:degreeconstraint}. Recall
that $\al_i=-\om_{i-1}+2\om_i-\om_{i+1}$ with the 
convention $\om_0=\om_{m+2}=0$, which gives 
\begin{equation}\label{eq:lambda12intermsofalphas}
\la_1 =\sum_{i=1}^{m+1}\al_i, \qquad 
\la_2 = \al_1+\al_{m+1} + 2\sum_{i=2}^m \al_i.
\end{equation}
Note that for weights $\eta\in P_G(\la_i)$, we have $\eta\preccurlyeq\la_i$,
or $\la_i -\eta \in Q^+_G$, so that the coefficient of $\al_1$
(or $\al_{m+1}$) in $\eta$ is less than or equal to $1$. Moreover, we 
see from \eqref{eq:lambda12intermsofalphas} that 
for $\la\in P^+_G(0)$ the degree $|\la|$ is equal to 
the coefficient of $\al_1$ (or $\al_{m+1}$) in $\la$. 

\begin{prop}\label{prop:cond3issatisfied}
For $\mu=a\om_1+b\om_2 \in P^+_K$, $a\in \Z$, $b\in \Z$, 
Condition \ref{cond:degreeconstraint} satisfied. 
\end{prop}

\begin{proof} 
We first assume $b\in \N$. Let $\eta\in P_G(\la_i)$, and 
assume $\nu_i+\eta = \nu_j +\la_\sph\in P^+_G(\mu)$ with 
$\la_\sph\in P^+_G(0)$. Then 
\begin{gather*}
\la_\sph-\eta = \nu_i-\nu_j = (i-j)(\om_2+\om_{m+1}-\om_1).
\end{gather*}
Since $\om_2+\om_{m+1}-\om_1=\sum_{k=2}^{m+1}\al_k$, we see that 
in the expansion of simple roots, the coefficient of $\al_1$ in 
$\la_\sph$ equals the coefficient of $\al_1$ in $\eta$, which is 
less than or equal to $1$. Since this coefficient is nonnegative and 
since $|\la_\sph|$ is the coefficient of $\al_1$, we 
get that $|\la_\sph|\leq 1$. This proves Condition \ref{cond:degreeconstraint} in case $b\in \N$. By duality it follows for $b\leq -a$. 

In case $-a<b<0$, Proposition \ref{prop:P+Gmu} gives two possible 
forms for $\nu_i$ and $\nu_j$. In case they have the same form, a similar argument as above proves $|\la_\sph|\leq 1$. Assume 
$\nu_i = (a-i) \om_1 + (-i-b)\om_m  + i\om_{m+1}$ for $0\leq i \leq -b$
and 
$\nu_j=(a-j) \om_1 + (b+j)\om_2  + j\om_{m+1}$ for $-b\leq j\leq a$. 
Then 
\begin{gather*}
\la_\sph-\eta = \nu_i-\nu_j = 
(j-i)\om_1 - (b+j)\om_2 - (b+i) \om_m +(i-j)\om_{m+1} \\
-(b+i)(\om_1+\om_m-\om_{m+1}) - (b+j)(-\om_1+\om_2+\om_{m+1})
\end{gather*}
and now we use additionally $\om_1+\om_m-\om_{m+1}=\sum_{k=1}^m \al_k$.
Since $-(b+i)\geq 0$ and $-(b+j)\leq 0$, the coefficient of 
$\al_{m+1}$ in $\la_\sph-\eta = \nu_i-\nu_j$ is nonpositive. 
Since the coefficient of $\al_{m+1}$ in $\eta$ is at most $1$, 
the coefficient of $\al_{m+1}$ in $\la_\sph$ is at most $1$, 
and thus $|\la_\sph|\leq 1$. The other situation, with the form 
of the $\nu_i$ and $\nu_j$ interchanged, can be proved 
analogously.
\end{proof}

Note that we have the following 
corollary of the proof of Proposition 
\ref{prop:cond3issatisfied}.

\begin{corollary}\label{cor:prop:cond3issatisfied}
For $\mu=a\om_1+b\om_2 \in P^+_K$, $a\in \N$, $b\in \N$, 
we have $\nu_i(\mu)\succ \nu_j(\mu)$ for $i>j$. 
\end{corollary}

A similar statement holds for $b\leq -a$, but not for 
$-a<b<0$. Then not all elements can be compared in the partial ordering. 

The reduced Weyl group $W=N_K(A)/M$ acts on the $M$-types in $V^K_\mu$. Let $n_w\in N_K(A)$ be a representative of 
$w$, then 
\eqref{eq:invariancepropMSF} shows 
\begin{equation*}
\Phi^\mu_\la( a_{w\bt}) = \Phi^\mu_\la( n_wa_{\bt}n_w^{-1}) =  
\pi^K_\mu(n_w) \Phi^\mu_\la(a_\bt) \pi^K_\mu(n_w^{-1}) 
\in \End_M(V^K_\mu).
\end{equation*}
For $T\in\End_M(V^K_\mu)$, the action $w\cdot T= 
\pi^K_\mu(n_w) T \pi^K_\mu(n_w^{-1})$ is well-defined, and preserves
orthogonal projections, and so it induces an action of $W$ on the $M$-types in $V^K_\mu$. In this case, the decomposition 
\eqref{eq:decompVKmuinMreps} splits into $1$-dimensional 
$M$-representations. From \eqref{eq:defn1n2forW}, we see that 
$s_2\in W$ acts trivially on the $M$-types, since it commutes
with $M$. For $s_1$ we see that it acts on the characters 
as $s_1\cdot \eta_1= \eta_2-\eta_1$, 
$s_1\cdot \eta_2= \eta_2$ leading to 
\begin{equation}\label{eq:actionWonMtypes}
s_1\cdot \si_k = \si_{a-k}, \qquad s_2 \cdot \si_k = \si_k
\end{equation}

\section{Special cases}\label{sec:specialcases}

In this section we give the simplest cases 
of embedding of $K$-representations in tensor products of
$G$-representations in order to
obtain the leading term. The first case concerns
the zonal spherical functions for the weights
$\la_1$ and $\la_2$ generating the spherical weights.
This is based on suitable embeddings of 
the $K$-fixed vector in a two fold tensor product. 
Next we find the embedding for the fundamental $K$-representation $V^K_{\om_1}$ in a two fold tensor product of $G$-representations. This will be used 
in Section \ref{sec:leadingcoefQnu} to obtain the 
leading terms of special matrix spherical functions. 

We first prove Lemma \ref{lem:tensorproductoffundrepr} which we use on several occassions. 

\begin{lemma}\label{lem:tensorproductoffundrepr}
For $i\leq j$ we have 
\[
V^G_{\om_i} \otimes V^G_{\om_j} \cong 
\bigoplus_{r=0}^{\min(i,m+2-j)} V^G_{\om_{i-r}+\om_{j+r}}
\]
with the convention $\om_0=0=\om_{m+2}$. 
\end{lemma}

\begin{proof} Observe that the fundamental weights for 
the root system of type $\mathrm{A}$ are minuscule weights, 
see e.g. \cite[p.~230]{Bour}, and for this case one has 
the multiplicity free decomposition
\[
V^G_{\om_i} \otimes V^G_{\om_j} \cong 
\bigoplus_{w\in W/W_{\om_j};\ \om_i+w\om_j\in P^+_G}
V^G_{\om_i+w\om_j},
\]
where $W=S_{m+2}$ is the Weyl group for $G$ and 
$W_{\om_j}= \{ w\in W \mid w\om_j=\om_j\} = S_j\times S_{m+2-j}$ 
is the stabiliser subgroup, see e.g.
\cite[Prop.~1]{Kass}, \cite[Cor. 3.5]{Kuma}. Since 
$W_{\om_j}$ is a parabolic subgroup, we can take coset representatives
of minimal length \cite[\S 1.10]{Hump}. Such an element is 
determined by a sequence $k_1<k_2<\cdots<k_j$ of numbers from 
$\{1,2,\cdots, m+2\}$ and defined by $w(j)=k_j$ and extended such that $w$ has minimal length. Using the expression for $\om_j$ as in 
\cite[Planche I]{Bour}, we get $w\om_j = \sum_{p=1}^j \om_{k_p}- \om_{k_p-1}$. It remains to determine the choices leading to 
$\om_i + w\om_j\in P^+_G =\bigoplus_{i=1}^{m+1}\N \om_i$. It
follows that the sequence $\{k_1,k_2,\cdots,k_j\}$ can have at most
on e
hole. Keeping track of these possibilities yields the result. 
\end{proof}

\subsection{Spherical functions on $A$}\label{ssec:sphericalfunctions}
We first construct explicit generators for the 
algebra of spherical functions for $(G,K)$. 
The natural representation $V^G_{\om_1}=V=\C^{m+2}$ of $G$ is equipped with 
the standard orthonormal basis $(e_1,\cdots, e_{m+2})$. 
Recall that $V^G_{\om_j} \cong \Lam^jV$.

\begin{lemma}\label{lem:explicitvKfixedvectorpsi1}
$V^G_{\om_1}\otimes V^G_{\om_{m+1}} \cong V^G_{\la_1}\oplus V^G_0$ and  define 
\[
v_1= e_1\otimes e_2\wedge e_3 \wedge \cdots \wedge e_{m+2} 
- e_2 \otimes e_1\wedge e_3 \wedge \cdots \wedge e_{m+2} 
\in V^G_{\om_1}\otimes V^G_{\om_{m+1}}.
\]
Then $v_1$ is a $K$-invariant vector, i.e. 
$v_1$ is contained in the $2$-dimensional space 
$(V^G_{\la_1})^K\oplus (V^G_0)^K$ and 
$v_1$ has a nonzero component in $(V^G_{\om_1+\om_{m+1}})^K=(V^G_{\la_1})^K$. 
\end{lemma}

\begin{proof} 
The tensor product decomposition follows from 
Lemma \ref{lem:tensorproductoffundrepr}. 
From \eqref{eq:defla1la2} we know that $0, \la_1=\om_1+\om_{m+1}\in P^+_G(0)$, so that 
$(V^G_{\om_1+\om_{m+1}})^K$ and $(V^G_0)^K$ are $1$-dimensional. 
It is a straightforward calculation to check that $v_1$ is a 
$K$-fixed vector, and the easiest way is to check 
that $E_i\cdot v_1 =0$, $i\in \{1,\ldots, m+1\}\setminus\{2\}$ and 
$H_i\cdot v_1= 0$, $i\in \{1,2,\ldots, m+1\}$. Note that 
$E_2\cdot v_1 \not=0$, so that $v$ is not contained in $(V^G_0)^K\cong \C$, 
and so has a nonzero component in $(V^G_{\om_1+\om_{m+1}})^K$. 
\end{proof}

Having $v_1$ given explicitly in Lemma \ref{lem:explicitvKfixedvectorpsi1} 
we can calculate the corresponding matrix entry restricted to $A$ explicitly using $a_\bt$ in \eqref{eq:defat}, and we obtain
\begin{equation*}
\langle (\pi^G_{\om_1}\otimes \pi^G_{\om_{m+1}})(a_\bt)v_1,v_1\rangle 
= \cos^2t_1 + \cos^2 t_2
\end{equation*}

\begin{lemma}\label{lem:psi1intermsofphi1}
Define $\psi_1\colon A \to \C$, $\psi_1(a_\bt)=\cos^2t_1 + \cos^2 t_2$
and let $\phi_1\colon A \to \C$ be the spherical function associated
to $V^G_{\la_1}$, then there exists a positive constant $\xi^1_1$ and a nonnegative constant $\xi^0_1$, so that 
$\psi_1=\xi^1_1\phi_1+\xi^0_1$ as functions on $A$. 
\end{lemma}

The constants $\xi^1_1$, $\xi^0_1$ can be calculated explicitly,
see Lemma \ref{lem:phi1phi2explicitpsi1psi2}. Moreover, we can also consider the identity as an identity for functions on $G$ by interpreting the matrix entries as functions on $G$. 

\begin{proof} Put $(V^G_{\la_1})^K=\C\hat{v}$, $\|\hat{v}\|=1$, and let $V^G_0=\C\tilde v$, then $\phi_1(a_\bt)= \langle \pi^G_{\la_1}(a_\bt)\hat{v}, \hat{v}\rangle$. In Lemma \ref{lem:explicitvKfixedvectorpsi1} we 
see that $v_1=a\hat{v}+b\tilde{v}$ with $0\not= a\in \C$. 
Then 
\begin{gather*}
\psi_1(a_\bt) = \langle (\pi^G_{\om_1}\otimes \pi^G_{\om_{m+1}})(a_\bt)v,v\rangle = 
|a|^2 \langle \pi^G_{\la_1}(a_\bt)\hat{v}, \hat{v}\rangle 
+|b|^2 \langle \pi^G_{0}(a_\bt)\tilde{v}, \tilde{v}\rangle 
= |a|^2 \phi_1(a_\bt) + |b|^2. 
\qedhere
\end{gather*}
\end{proof}

In order to find the second spherical function, we proceed similarly.

\begin{lemma}\label{lem:explicitvKfixedvectorpsi2}
$V^G_{\om_2}\otimes V^G_{\om_{m}} \cong V^G_{\la_2} \oplus 
V^G_{\la_1}\oplus V^G_0$ and  define 
\[
v_2= e_1\wedge e_2 \otimes  e_3 \wedge \cdots \wedge e_{m+2} 
\in V^G_{\om_2}\otimes V^G_{\om_{m}}.
\]
Then $v_2$ is a $K$-invariant vector, and $v_2$ has a nonzero 
component in $(V^G_{\la_2})^K$. Moreover, 
\[
\psi_2 \colon A \to\C, \qquad 
\psi_2(a_\bt) =  
\langle (\pi^G_{\om_2}\otimes \pi^G_{\om_{m}})(a_\bt)v_2,v_2\rangle 
= (\cos t_1)^2 (\cos t_2)^2
\]
and $\psi_2 = \xi^2_2\phi_2+\xi^1_2\phi_1+\xi^0_2$, where $\phi_2$ is the spherical function corresponding to $\la_2\in P^+_G(0)$ and the 
constants $\xi^2_2>0$ and $\xi^1_2$ and $\xi^0_2$ are nonnegative. 
\end{lemma}

The proof of Lemma \ref{lem:explicitvKfixedvectorpsi2} follows
the line of proof of Lemma \ref{lem:explicitvKfixedvectorpsi1} and 
Lemma \ref{lem:psi1intermsofphi1}. 
It is possible to calculate the constants $\xi^i_2$ explicitly,
see Lemma \ref{lem:phi1phi2explicitpsi1psi2}.

\begin{proof} The tensor product decomposition follows from 
Lemma \ref{lem:tensorproductoffundrepr}. 
The $K$-invariance of $v_2$ follows from 
$E_i\cdot v_2 =0$, $i\in \{1,\ldots, m+1\}\setminus\{2\}$ and 
$H_i\cdot v_2 = 0$, $i\in \{1,\ldots, m+1\}$, which follows 
straightforwardly. Then the matrix entry can be calculated 
using \eqref{eq:defat}, and this gives the statement of the 
explicit expression 
for $\psi_2(a_\bt)$. Since $v_2$ is a linear combination of the 
$K$-fixed vectors of $V^G_{\la_2}$, $V^G_{\la_1}$ and $V^G_0$,
we find analogously that $\psi_2$ is a linear combination of 
$\phi_1$, $\phi_2$ and the constant with nonnegative coefficients. 
Since the function $(\cos t_1)^2 (\cos t_2)^2$ is not a linear 
combination of $(\cos t_1)^2+ (\cos t_2)^2$ and the constants, the 
coefficient of $\phi_2$ has to be nonzero. 
\end{proof}

\begin{remark}\label{rmk:actionAcapMonsphericalfunctions}
Note that $A\cap M \cong \Z/2\Z \times \Z/2\Z$, and 
so the spherical functions, satisfying $\phi(ma_\bt)=\phi(a_\bt)$
for $m\in A\cap M$, show that the spherical 
functions in Lemmas \ref{lem:explicitvKfixedvectorpsi1} and  \ref{lem:explicitvKfixedvectorpsi2}, have to be invariant 
under $(\cos t_1, \cos t_2) \mapsto (\pm \cos t_1,\pm \cos t_2)$
for all choices of signs. 
\end{remark}

\subsection{The special case $\mu=\om_1$}\label{ssec:specialcasemu=om1}
From Proposition \ref{prop:P+Gmu} we know that 
$B(\om_1)$ consists of $\om_1$ and $\om_2+\om_{m+1}$. The 
$K$-equivariant map $V^K_{\om_1} \to V^G_{\om_1}$ is the 
standard embedding, which sends the highest weight vector for $K$ to the 
highest weight vector for $G$ in the natural representation. We need
to understand the 
$K$-equivariant map $V^K_{\om_1} \to V^G_{\om_2+\om_{m+1}}$, and 
it suffices to understand the $K$-highest weight vector in 
$V^G_{\om_2+\om_{m+1}}$. We proceed as in Subsection \ref{ssec:sphericalfunctions}. 

\begin{lemma}\label{lem:explicitvKhwvectorom1}
$V^G_{\om_2}\otimes V^G_{\om_{m+1}} \cong 
V^G_{\om_2+\om_{m+1}} \oplus V^G_{\om_1}$
and the vector
\[
v_0 = e_1\wedge e_2\otimes e_1\wedge e_3 \wedge 
\cdots \wedge e_{m+2} \in 
V^G_{\om_2}\otimes V^G_{\om_{m+1}}
\]
is a $K$-highest weight vector of weight $\om_1$.
The vector $v_0$ has a nonzero component in 
$V^G_{\om_2+\om_{m+1}}$. 
\end{lemma}

It follows from the tensor product decomposition 
and Proposition \ref{prop:P+Gmu} for $\mu=\om_1$ that there
is a $2$-dimensional space of 
$K$-highest weight vectors of weight $\om_1$. It is 
possible to explicitly write down a linearly independent
vector and give the $K$-highest weight vectors of weight $\om_1$ 
in $V^G_{\om_2+\om_{m+1}}$ and $V^G_{\om_1}$. 

\begin{proof} Lemma \ref{lem:tensorproductoffundrepr} 
proves the first statement. Note that both representations 
in the direct sum correspond to $B(\om_1)=\{\om_1, \om_2+\om_{m+1}\}$. 
It is straightforward to check that 
$E_i\cdot v_0=0$, $i\in \{1,\ldots, m+1\}\setminus\{2\}$ and 
$H_1\cdot v_0=v_0$, 
$H_i\cdot v_0= 0$, $i\in \{2,\ldots, m+1\}$, so that 
$v_0$ is a $K$-highest weight vector of weight $\om_1$. 
Note that the $K$-highest weight vector of weight $\om_1$ 
in $V^G_{\om_1}$ is also a $G$-highest weight vector of weight $\om_1$,
but $E_2\cdot v_0\not=0$.  
So the vector $v_0$ has a nonzero component in 
$V^G_{\om_2+\om_{m+1}}$.  
\end{proof}

\section{The leading term 
of matrix spherical functions for $B(\mu)$}
\label{sec:leadingcoefQnu}

We focus on the case $\mu=a\om_1+b\om_2$, with $b\in \N$, and then discuss the 
case $b<0$ briefly in Section \ref{sec:bnegative}. 
In this case $\nu_i=(a-i)\om_1+i(\om_2+\om_{m+1})+b\om_2$,
$0\leq i \leq a$, see Proposition \ref{prop:P+Gmu}. 
Instead of trying to determine the $K$-equivariant embedding 
$V^K_\mu \to V^G_{\nu_i}$, we embed $V^K_\mu$ in a much 
bigger $G$-representation containing $V^G_{\nu_i}$ in which
we can identify a $K$-highest weight of weight $\mu$ that 
`sees' $V^G_{\nu_i}$, i.e. has a nonzero component in 
$V^G_{\nu_i}$. 

Recall that the representation $V^G_{N\om_1}$ can be realised 
in the space of polynomials in variables $(x_1,x_2,\cdots, x_{m+2})$ 
which are homogeneous of degree $N$. Its $G$-highest weight 
vector is $x_1^N$. Now define the tensor product representation with 
specific element $u$; 
\begin{equation}\label{eq:defUGnuirepr}
\begin{split}
U^G_{\nu_i} &= V^G_{(a-i)\om_1} \otimes 
\Bigl( V^G_{\om_2}\otimes V^G_{\om_{m+1}}\Bigr)^{\otimes i}
\otimes \Bigl( V^G_{\om_2}\Bigr)^{\otimes b} \\
u&= x_1^{a-i} \otimes 
\underbrace{v_0\otimes \cdots \otimes v_0}_{i\ \mathrm{times}} \otimes 
\underbrace{e_1\wedge e_2 \otimes \cdots \otimes e_1\wedge e_2 }_{b\ \mathrm{times}},
\end{split}
\end{equation}
where $v_0$ is as in Lemma \ref{lem:explicitvKhwvectorom1}.  
Then $u$ is a $K$-highest weight vector of weight 
$\mu=a\om_1+b\om_2$ by Lemma \ref{lem:explicitvKhwvectorom1}, since 
$e_1\wedge e_2\in V^G_{\om_2}$ is the $G$- and $K$-highest weight vector of weight $\om_2$. 
Moreover, 
\begin{equation}\label{eq:decompositiondefUGnuirepr}
U^G_{\nu_i} = V^G_{\nu_i} \oplus \bigoplus_{\la \prec \nu_i} n_\la
V^G_\la,
\end{equation}
for certain multiplicities $n_\la$. Since we are only interested in 
$\la\in P^+_G(\mu)$,  we need Lemma \ref{lem:lainPGmusallerthannui}.

\begin{lemma}\label{lem:lainPGmusallerthannui}
Let $\mu=a\om_1+b\om_2$ with $b\in \N$. Then 
$\{\la\in P^+_G(\mu)\mid \la \preccurlyeq \nu_i\} =
\{\nu_0, \cdots, \nu_i\}$. 
\end{lemma}

\begin{proof} Using the ideas and identities as in 
Subsection \ref{ssec:condition3} we assume 
$\nu_j+n_1\la_1+n_2\la_2\preccurlyeq \nu_i$, $n_1,n_2\in \N$. 
Writing 
\begin{gather*}
\nu_i-(\nu_j+n_1\la_1+n_2\la_2) = 
(i-j)(-\om_1+\om_2+\om_{m+1})-n_1(\om_1+\om_{m+1})-n_2(\om_2+\om_m) \\
= (-n_1-n_2)\al_1 +(i-j-n_1-n_2)\al_{m+1} 
+(i-j-n_1-2n_2)\sum_{k=2}^m\al_k
\end{gather*}
we see that this is in $Q^+_G$ if and only if $n_1=n_2=0$ and $i\geq j$.
\end{proof}

Our next objective is to give an explicit expression for the 
matrix valued spherical function associated to the 
$K$-equivariant embedding $V^K_\mu\to U^G_{\nu_i}$, which 
maps the highest weight vector of $V^K_\mu$ to $u$. 
In order to describe the result we need the 
Krawtchouk polynomials, see e.g. \cite[\S 6.2]{Isma}, 
\cite[\S 9.11]{KoekLS}. The Krawtchouk polynomials are defined
as a terminating hypergeometric series and are generated by a generating
function:
\begin{equation}\label{eq:defKrawtchouk-generating}
\begin{split}
K_n(x;p,N) &= \rFs{2}{1}{-n,-x}{-N}{\frac{1}{p}}, 
\qquad N\in \N, \quad x,n\in \{0,1,\ldots, N\} \\
\sum_{n=0}^N &\binom{N}{n} K_n(x;p,N) t^n = (1-\frac{1-p}{p}t)^x (1+t)^{N-x}.
\end{split}
\end{equation}
Note that the Krawtchouk polynomials are self-dual:
$K_n(x;p,N)= K_x(n;p,N)$, and that 
$K_0(x;p,N)=1=K_n(0;p,N)$. 

\begin{prop}\label{prop:matrixentriesatinUnui}
Let $\mu=a\om_1+b\om_2$, $b\in \N$, 
$\nu_i= (a-i)\om_1+i(\om_2+\om_{m+1})+b\om_2$, and 
$U^G_{\nu_i}$ the representation defined in \eqref{eq:defUGnuirepr}.
Then for $k,l\in \{0,1,\ldots, a\}$ 
\begin{gather*}
\langle \pi_{U^G_{\nu_i}}(a_\bt) F_1^k\cdot u, F_1^l\cdot u \rangle
= \de_{k,l} \| F_1^k\cdot u\|^2 
(\cos t_1)^{a+b-k} (\cos t_2)^{b+2i+k} 
K_i(k; \frac{\cos^2 t_2}{\cos^2 t_2-\cos^2 t_1},a)
\end{gather*}
with $a_\bt$ as in \eqref{eq:defat} and the Krawtchouk 
polynomials as in \eqref{eq:defKrawtchouk-generating}. 
\end{prop}

\begin{remark}\label{rmk:prop:matrixentriesatinUnui}
(i) The fact that we get zero for $k\not=l$ follows from the fact that
matrix spherical functions restricted to $A$ are $M$-intertwiners 
and the vectors $F^k\cdot u$ correspond to different $M$-types for different $k$. Indeed, $u$ spans a one-dimensional $M$-representation
of weight $\si_0(\mu)=a \eta_1+b\eta_2$ by \eqref{eq:defUGnuirepr}
and Lemma \ref{lem:explicitvKhwvectorom1}, and more generally  
$F^k\cdot u$ corresponds to the one-dimensional $M$-representation
of weight $\si_k(\mu)=(a-2k) \eta_1+ (b+k)\eta_2$, 
see \eqref{eq:decompVKmuinMreps} for the $K$-representation generated
by $u$. 

(ii) For $k=l$, the right hand side is a polynomial in $(\cos t_1, \cos t_2)$, and it is a homogeneous polynomial of degree $a+2b+2i$
in $(\cos t_1, \cos t_2)$. Note that the degree of
homogeneity is independent of $k$. Indeed, for $k=l$ the right hand
side of Proposition \ref{prop:matrixentriesatinUnui} equals 
\[
\| F_1^k\cdot u\|^2 
(\cos t_1)^{a+b-k}  
\sum_{p=0}^{\min(i,k)} \frac{(-i)_p(-k)_p}{p!\, (-a)_p}
(\cos^2 t_2-\cos^2 t_1)^p (\cos t_2)^{b+2i+k-2p}
\]
using \eqref{eq:defKrawtchouk-generating} and the  standard notation for Pochhammer symbols $(x)_p =\prod_{i=0}^{p-1}(x+i)$, see e.g. 
\cite{AndrAR}, \cite{Isma}, \cite{KoekLS}. 

(iii) Using $\Phi(m a_\bt)= \pi^K_\mu(m)\Phi(a_\bt)$ for 
$m\in A\cap M$ and the decomposition \eqref{eq:decompVKmuinMreps}
and $\si_k(\diag(\zeta_1,\zeta_2,1,\cdots,1, \zeta_2,\zeta_1))
= \zeta_1^{a+b-k} \zeta_2^{b+k}$ for $\zeta_i\in \Z/2\Z$, we see that the right hand side 
has to be invariant up to $(-1)^{a+b-k}$ under $\cos t_1\mapsto -\cos t_1$ and invariant up to $(-1)^{b+k}$ under
$\cos t_2\mapsto -\cos t_2$, 
cf. Remark \ref{rmk:actionAcapMonsphericalfunctions}. This also follows directly from the explicit expression of 
Proposition \ref{prop:matrixentriesatinUnui}. 
\end{remark}

\begin{proof} We put 
$a_\bt(r,s) = \exp(sE_1)a_\bt\exp(rF_1)$, then using the unitarity
of the representation $U^G_{\nu_i}$ we obtain 
\begin{equation}\label{eq:prop:matrixentriesatinUnui1}
\langle \pi_{U^G_{\nu_i}}(a_\bt) F_1^k\cdot u, F_1^l\cdot u \rangle
= \frac{\partial^k}{\partial r^k}\big\vert_{r=0}
\frac{\partial^l}{\partial s^l}\big\vert_{s=0}
\langle \pi_{U^G_{\nu_i}}(a_\bt(r,s)) u, u \rangle.
\end{equation}
Now 
\begin{gather*}
a_\bt(r,s) = 
\begin{pmatrix} A & 0 & B \\ 0 & I & 0 \\ C & 0 & D \end{pmatrix}, 
\qquad 
A= \begin{pmatrix} \cos t_1+ rs \cos t_2 & s\cos t_2 \\ r \cos t_2&  \cos t_2\end{pmatrix},\\
B = \begin{pmatrix}  si \sin t_2 & i \sin t_1 \\ i \sin t_2 & 0 \end{pmatrix}, \quad
C = \begin{pmatrix}  ri \sin t_2 & i \sin t_2 \\ i \sin t_1 & 0 \end{pmatrix}, \quad 
D = \begin{pmatrix}  \cos t_2 & 0 \\ 0 & \cos t_1 \end{pmatrix}
\end{gather*}
and we can calculate the action of $a_\bt(r,s)$ on each of the factors
in $u\in U^G_{\nu_i}$. We get 
\begin{gather*}
\langle a_\bt(r,s) \cdot x_1^{a-i}, x_1^{a-i}\rangle =
(\cos t_1 + rs \cos t_2)^{a-i} \langle x_1^{a-i}, x_1^{a-i}\rangle \\
\langle a_\bt(r,s) \cdot v_0, v_0 \rangle =
\cos t_1 \cos t_2 (\cos t_2 + rs \cos t_1) \langle  v_0, v_0 \rangle \\
\langle a_\bt(r,s) \cdot e_1\wedge e_2, e_1\wedge e_2 \rangle =
\cos t_1 \cos t_2 \langle e_1\wedge e_2, e_1\wedge e_2 \rangle
\end{gather*}
and this gives
\begin{equation}\label{eq:prop:matrixentriesatinUnui1-ref}
\langle \pi_{U^G_{\nu_i}}(a_\bt(r,s)) u, u \rangle = 
(\cos t_1 + rs \cos t_2)^{a-i} 
(\cos t_2 + rs \cos t_1)^i 
(\cos t_1 \cos t_2)^{b+i} \langle u, u \rangle.
\end{equation}
The first two factors can be expanded in terms of 
Krawtchouk polynomials using the generating function of 
\eqref{eq:defKrawtchouk-generating}, and this gives
\begin{gather*}
\frac{\langle \pi_{U^G_{\nu_i}}(a_\bt(r,s)) u, u \rangle}{\langle u, u\rangle} = 
(\cos t_1)^{a+b} (\cos t_2)^{b+2i} 
\sum_{n=0}^a \binom{a}{n} K_n(i; \frac{\cos^2t_2}{\cos^2 t_2-\cos^2t_1},a)
\Bigl( rs \frac{\cos t_2}{\cos t_1}\Bigr)^n
\end{gather*}
Now the statement of the proposition follows using \eqref{eq:prop:matrixentriesatinUnui1}.
\end{proof}

\begin{remark}\label{rmk:prop:matrixentriesatinUnui-2}
Note that the right hand side of \eqref{eq:prop:matrixentriesatinUnui1-ref}
is a polynomial of the product $rs$. This follows from 
\eqref{eq:prop:matrixentriesatinUnui1} being zero for $k\not=l$, and 
this follows from the fact that $a_\bt$ commutes with $M$ and 
$F^k\cdot u$ and $F^l\cdot u$ realise different one-dimensional 
$M$-representations for $k\not=l$, cf. Remark  
\ref{rmk:prop:matrixentriesatinUnui}(i). 
\end{remark}

We can now collect the results of this section 
into Theorem \ref{thm:Qmunuisasmatrixspherfunction}. 

\begin{theorem}\label{thm:Qmunuisasmatrixspherfunction} 
Let $\mu=a\om_1+b\om_2$, $a,b\in \N$, and let 
$\nu_i = (a-i)\om_1+i(\om_2+\om_{m+1})+b\om_2 \in B(\mu)$, 
$i\in \{0,\ldots, a\}$. 
Let $v_\mu$ be the highest weight vector of $V^K_\mu$, and 
define $j\colon V^K_\mu \to U^G_{\nu_i}$ to be the $K$-equivariant
map sending $v_\mu \mapsto u$. Then 
\begin{gather*}
Q^\mu_{\nu_i} \colon G \to \End(V^K_\mu),
\quad g\mapsto j^\ast \circ \pi_{U^G_{\nu_i}}(g) \circ j
\end{gather*}
is a matrix spherical function and restricted to $A$ we have
\begin{gather*}
Q^\mu_{\nu_i}(a_\bt) \bigl(F_1^k\cdot v_\mu\bigr) =
q^\mu_{\nu_i, \si_k} (a_\bt)\, F_1^k\cdot v_\mu, \\
q^\mu_{\nu_i, \si_k} (a_\bt) =
(\cos t_1)^{a+b-k} (\cos t_2)^{b+2i+k} 
K_i(k; \frac{\cos^2 t_2}{\cos^2 t_2-\cos^2 t_1},a) 
\, .
\end{gather*}
Moreover, as matrix spherical functions on $G$ we have
\[
Q^\mu_{\nu_i} = \sum_{r=0}^i a^i_r \, \Phi^\mu_{\nu_r}, 
\qquad a^i_r\in \C, \quad a^r_r\not=0.
\]
\end{theorem}

So we see that the transition of the elements $\Phi^\mu_{\nu_i}$, 
$i\in \{0,\ldots, a\}$ to $Q^\mu_{\nu_i}$, 
$i\in \{0,\ldots, a\}$ is given by a triangular matrix with nonzero diagonal entries. Hence, $(Q^\mu_{\nu_i})_{i=0}^a$ and 
$(\Phi^\mu_{\nu_i})_{i=0}^a$ span the same space of matrix spherical functions, from which the matrix part $W$ of the weight 
as in \eqref{eq:MatrixOPorthogonality} can be obtained. 

\begin{corollary}\label{cor:thm:Qmunuisasmatrixspherfunction} 
For $i\in \{0,\ldots, a\}$ we have 
$\Phi^\mu_{\nu_i} = \sum_{r=0}^i d^i_r Q^\mu_{\nu_r}$ 
with $d^i_r\in \C$ and $d^r_r\not=0$. 
\end{corollary}

Corollary \ref{cor:thm:Qmunuisasmatrixspherfunction}  and 
the degree consideration of Remark \ref{rmk:prop:matrixentriesatinUnui} 
motivate to consider the explicit matrix spherical function 
$Q^\mu_{\nu_i}$ as the leading term of the 
matrix spherical function $\Phi^\mu_{\nu_i}$ of Subsection \ref{ssec:generalsetup}. Note that $d^r_r=(a_r^r)^{-1}$

\begin{proof}[Proof of Theorem \ref{thm:Qmunuisasmatrixspherfunction}]
The first statement follows from the general set-up in 
Subsection \ref{ssec:generalsetup} and Proposition 
\ref{prop:matrixentriesatinUnui}. For the last statement we
recall that $\{ \Phi^\mu_\la\mid \la \in P^+_G(\mu)\}$ form a 
basis for the matrix spherical functions, see Subsection \ref{ssec:generalsetup}. By 
\eqref{eq:decompositiondefUGnuirepr} and Lemma \ref{lem:lainPGmusallerthannui} we find that the only 
matrix spherical functions of type $\mu$ occurring in $U^G_{\nu_i}$ are
$\Phi^\mu_{\nu_r}$, $r\in \{0,\ldots, i\}$.  
It remains to show that $a_r^r\not=0$. 

In case $r=0$ we have $Q^\mu_{\nu_0}=\Phi^\mu_{\nu_0}$ since both are 
the identity in $\End(V^K_\mu)$ for the identity in $G$, so $a^0_0=1$.
Assume that $a^i_i\not=0$ for $i\in \{0, \ldots, r-1\}$, $1\leq r\leq a$, and 
$a^r_r=0$. We show that this leads to a contradiction. Indeed,
then $Q^\mu_{\nu_r}$ can be expressed in terms of
$\Phi^\mu_{\nu_j}$, $j<r$, which in turn can be expressed in terms 
of $Q^\mu_{\nu_j}$, $j<r$. Hence, there is a nontrivial linear
dependence between the matrix spherical functions 
$\sum_{j=0}^r c_j Q^\mu_{\nu_j} =0$. Evaluating at $a_\bt$, acting
on the $K$-highest weight vector $v_\mu\in V^K_\mu$ and taking inner products with $v_\mu$ and using the 
first part of the theorem, i.e. Proposition \ref{prop:matrixentriesatinUnui}, we get a nontrivial linear 
dependence of the form 
\[
\sum_{j=0}^r c_j (\cos t_1)^{a+b} (\cos t_2)^{b+2j} =0, \qquad 
\forall t_1, t_2.
\]
This is the required contradiction. 
\end{proof}

\section{The matrix weight}\label{sec:MatrixWeight}

We keep $\mu=a\om_1+b\om_2$ with $a,b\in \N$ fixed. Then we identify
$\End_M(V^K_\mu) \cong \C^{a+1}$
by Schur's Lemma and \eqref{eq:decompVKmuinMreps} and we set
\begin{equation}\label{eq:defphimulafromPhimula}
\phi^\mu_{\la, \si_k} (a_\bt) \colon A\to \C,
\qquad \Phi^\mu_\la(a_\bt) \vert_{V^M_{\si_k}} = 
\phi^\mu_{\la, \si_k} (a_\bt) \, \Id_{V^M_{\si_k}}
\end{equation}
for $\la\in P^+_G(\mu)$, $k\in \{0,\ldots, a\}$. 
Note that $W$-invariance leads to, see \eqref{eq:actionWonMtypes},  
\begin{equation}\label{eq:invpropqmularhokforW}
\phi^\mu_{\la, \si_k} (s_1 a_\bt)
= \phi^\mu_{\la, \si_{a-k}} (a_\bt), 
\qquad 
\phi^\mu_{\la, \si_k} (s_2 a_\bt)
= \phi^\mu_{\la, \si_k} (a_\bt)
\end{equation}
and similarly for $q^\mu_{\la,\si_k}(a_\bt)$ 
because of Theorem \ref{thm:Qmunuisasmatrixspherfunction}. The 
nontrivial action for $q^\mu_{\la,\si_k}(a_\bt)$ corresponds
to Pfaff's transformation formula for ${}_2F_1$-series 
\cite[Thm.~2.2.5]{AndrAR}. 

We define the lower triangular matrix $L$ by 
$L_{i,j}= d^i_j$, $j\leq i$, with $d^i_j$ as in 
Corollary \ref{cor:thm:Qmunuisasmatrixspherfunction}. 
Then $L$ is invertible. Upon defining the matrices 
$\Phi_0$ and $Q_0$ on $A$ by $(\Phi_0)_{i,j}= \phi^\mu_{\nu_i,\si_j}$ and $(Q_0)_{i,j}= q^\mu_{\nu_i, \si_j}$ we see that 
Corollary \ref{cor:thm:Qmunuisasmatrixspherfunction} can be 
rephrased as $\Phi_0=LQ_0$, and we calculate $L$ 
explicitly in Proposition \ref{prop:Phi0isLQ0}. 
Moreover, $\Phi_0(s_1a_\bt)= \Phi_0(a_\bt)J$, where 
$J_{i,j}=1$ if $i+j=a$ and $J_{i,j}=0$ otherwise, and similarly
$Q_0(s_1a_\bt)= Q_0(a_\bt)J$ by \eqref{eq:invpropqmularhokforW}. 

As a function on $A$ we see that 
the matrix weight $W$ in \eqref{eq:MatrixOPorthogonality} can 
be written as $\Phi_0\Phi_0^\ast$, for 
which each matrix entry is a polynomial in $(\phi_1,\phi_2)$. 
Note that the weight $W$ is a matrix function on $A$ which is invariant
for the action of the reduced Weyl group. 
We switch from the matrix weight $W$ on $A$ to the matrix weight 
$S=Q_0(Q_0)^\ast$, so that $W=LSL^\ast$ as functions on $A$ for the 
constant lower triangular matrix $L$. 
Note that $S$ as matrix function on $A$ is invariant for the action of the reduced Weyl group. Note that $S$ is a polynomial 
in $(\psi_1,\psi_2)$ and we have 
for the matrix entries $S_{i,j}$ of the weight $S$ 
\begin{gather}\label{eq:defSij}
S_{i,j}(\psi_1(a_\bt), \psi_2(a_\bt))= 
\sum_{k=0}^a q^\mu_{\nu_i,\si_k}(a_\bt) 
\overline{q^\mu_{\nu_j,\si_k}(a_\bt)} = \\
\sum_{k=0}^a
(\cos t_1)^{2a+2b-2k} (\cos t_2)^{2b+2k+2i+2j} 
K_k(i; \frac{\cos^2 t_2}{\cos^2 t_2-\cos^2 t_1},a) 
K_k(j; \frac{\cos^2 t_2}{\cos^2 t_2-\cos^2 t_1},a) \nonumber
\end{gather}
and by this expression we see that 
$S_{i,j}(\psi_1(a_\bt), \psi_2(a_\bt))$ is a homogeneous polynomial 
in $(\cos t_1, \cos t_2)$ of degree $2a+4b+2i+2j$. The simplest non-scalar cases for $a=1$ and $a=2$ give
the following expressions for $S(\psi_1,\psi_2)$ 
\begin{equation}\label{eq:Sfora=1}
\psi_2^b 
\begin{pmatrix}
\psi_1 & 2\psi_2 \\ 2\psi_2 & \psi_1\psi_2 
\end{pmatrix} 
\quad \text{and} \quad
\psi_{2}^{b}
\begin{pmatrix}
      \psi_{1}^{2}-\psi_{2} & \frac{3}{2}\psi_{1}\psi_{2} & 3\psi_{2}^{2} \\
      \frac{3}{2}\psi_{1}\psi_{2} & 2\psi_{2}^{2}+\frac{1}{4}\psi_{1}^{2}\psi_{2} & \frac{3}{2}\psi_{1}\psi_{2}^{2} \\
      3\psi_{2}^{2} & \frac{3}{2}\psi_{1}\psi_{2}^{2} & \psi_{2}^{2}(\psi_{1}^{2}-\psi_{2}) \\
    \end{pmatrix}.
\end{equation}
Note that in \eqref{eq:Sfora=1} the matrix part of $S$ is determined by 
$a$, and the $b$-dependence is only in the scalar part $\psi_2^b$. This
follows in general from \eqref{eq:defSij}.

\begin{prop}\label{prop:Sisindecomposable}
The matrix weight $S$ is indecomposable, i.e. 
\begin{gather*}
\cA = \{ T\in M_{a+1}(\C) \mid 
TS(\psi_1(a_\bt), \psi_2(a_\bt)) = 
S(\psi_1(a_\bt), \psi_2(a_\bt))T^\ast, \ \forall\, t_1,t_2\} = \R\Id, 
\\
\cA'=\{ T\in M_{a+1}(\C) \mid 
TS(\psi_1(a_\bt), \psi_2(a_\bt)) = 
S(\psi_1(a_\bt), \psi_2(a_\bt))T,\ \forall\, t_1,t_2\} = \C\Id. 
\end{gather*}
\end{prop}

\begin{remark}\label{rmk:prop:Sisindecomposable}
These notions of indecomposability of the matrix weight for 
multivariable weights have not yet been introduced, but it
follows the definition of the single variable case \cite{KoelR}, \cite{TiraZ}, which can be generalised directly. 
Note that $\cA'$, which is denoted $A$ in \cite{KoelR}, is a $\ast$-algebra, and $\cA$ is a real vectorspace. 
The corresponding vector spaces for the weight $W=LSL^\ast$ 
are then also trivial, which follows directly for $\cA$ and 
the invertibility of $L$. For $\cA'$ this follows from 
\cite[Thm~2.3]{KoelR}. 
\end{remark}

\begin{proof}
Recall that the degree of $S_{i,j}$ as a homogeneous 
polynomial in 
$(\cos t_{1},\cos t_{2})$ is $2a+4b+2i+2j$. Assume $T\in \cA'$ so that $ST=TS$. We consider the $(i,j)$th entry:
\begin{equation*}
\sum_{k=0}^a S_{i,k}(\cos t_{1},\cos t_{2})\,T_{k,j}=\sum_{r=0}^a T_{i,r}\, S_{r,j}(\cos t_{1},\cos t_{2}), \quad \forall\, t_1,t_2.
\end{equation*}
Consider this a polynomial identity in $(\cos t_{1},\cos t_{2})$ and consider the total degree of both sides. Assume that $i<j$, then we see that $T_{i,r}=0$ for $r>a+i-j$. Taking $j=a$, we see that 
$T_{i,r}=0$ for $r>i$. So $T$ is lower triangular. A similar 
deduction for $i>j$ shows that $T$ is upper triangular, and so
$T$ is diagonal. Then we obtain 
$S_{i,j}(\cos t_{1},\cos t_{2})\,T_{j,j}= T_{i,i}\, S_{i,j}(\cos t_{1},\cos t_{2})$, and since $S_{i,j}(\cos t_{1},\cos t_{2})$
is a nonzero function, we find $T_{i,i}=T_{j,j}$. So $T$ is a
multiple of the identity. 

Assume 
$T\in \cA$ so that $TS=ST^\ast$. We consider the $(i,j)$th entry:
\begin{equation*}
\sum_{k=0}^a S_{i,k}(\cos t_{1},\cos t_{2})\,\overline{T_{j,k}}=\sum_{r=0}^a T_{i,r}\, S_{r,j}(\cos t_{1},\cos t_{2}), \quad \forall\, t_1,t_2.
\end{equation*}
Arguing as in the previous case, we see that $i<j$ leads to 
$T$ being lower triangular. This gives
$\sum_{k=0}^j S_{i,k}(\cos t_{1},\cos t_{2})\,\overline{T_{j,k}}=\sum_{r=0}^i T_{i,r}\, S_{r,j}(\cos t_{1},\cos t_{2})$.
Considering the homogeneous part of highest degree $2a+4b+2i+2j$ 
gives $\overline{T_{j,j}}= T_{i,i}$, so that each diagonal entry is
equal to the same real number.
Next comparing the homogeneous part of the same
degree leads to 
$S_{i,k}(\cos t_{1},\cos t_{2})\,\overline{T_{j,k}}= 
T_{i,k+i-j}\, S_{k+i-j,j}(\cos t_{1},\cos t_{2})$, so that in 
case $j>i$ we get $\overline{T_{j,k}}=0$ for $0\leq k<j-i$. Taking $i=0$ shows that $T$ is upper triangular. Hence, $T$ is a real 
multiple of the identity. 
\end{proof}

Next we calculate the determinant of $S$. For this it suffices
to calculate the determinant of $Q_0$ for which we use the 
orthogonality properties of the Krawtchouk polynomials. Recall 
e.g. \cite[\S 6.2]{Isma}, \cite[\S 9.11]{KoekLS}, using the 
notation of \eqref{eq:defKrawtchouk-generating}, the 
orthogonality relations
\begin{gather}
\sum_{x=0}^N w(x;p,N) K_n(x;p,N) K_m(x;p,N)
 = \de_{m,n} h(n;p,N),
\label{eq:deforthoKrawtchouk} \\
w(x;p,N) = \binom{N}{x} p^x(1-p)^{N-x} , \qquad h(n;p,N)= 
\frac{(-1)^n n!}{(-N)_n} \left( \frac{1-p}{p}\right)^n, \nonumber
\end{gather}
which is a positive finite discrete measure for $0<p<1$. 
Rewriting shows that the matrix $B=\bigl( \frac{\sqrt{w(x;p,N)}}{\sqrt{h(n;p,N)}}K_n(x;p,N)\bigr)_{n,x=0}^N$ is an orthogonal 
matrix, so of determinant $\pm 1$. Writing $B$ as product of a 
diagonal matrix times the matrix with entries the Krawtchouk polynomials times a diagonal matrix, and introducing 
additional parameters gives 
\begin{equation}\label{eq:determinantKrawtchoukpols}
\det \bigl(t^n s^x K_n(x;p,N)\bigr)_{n,x=0}^N = 
\pm (st)^{\frac12N(N+1)}
\Bigl(\prod_{n=0}^Nh(n;p,N)\Bigr)^\frac12 
\Bigl(\prod_{x=0}^Nw(x;p,N)\Bigr)^{-\frac12}
\end{equation}

\begin{prop}\label{prop:detS} For $\mu=a\om_1+b\om_2$, $a,b\in \N$,  $a_\bt\in A$ we have 
\[
\det\bigl( S(a_\bt)\bigr) 
=\Bigl( \prod_{n=0}^a \binom{a}{n}\Bigr)^{-2}
\bigl(\cos t_1 \cos t_2\bigr)^{2b(a+1)}
\bigl( \cos t_1 \cos t_2 (\cos^2 t_1 - \cos^2 t_2)\bigr)^{a(a+1)}
\]
\end{prop}

\begin{proof} With 
$Q_0(a_\bt)_{i,j}= q^\mu_{\nu_i, \si_j}(a_\bt)$, $0\leq i,j\leq a$, expressed in Theorem \ref{thm:Qmunuisasmatrixspherfunction} 
in terms of Krawtchouk polynomials we take out the terms 
independent of $i,j$, and next we apply 
\eqref{eq:determinantKrawtchoukpols} to get 
\begin{gather*}
\det\bigl( Q_0(a_\bt)\bigr) 
= \pm \bigl( \cos^{a+b}t_1 \cos^b t_2\bigr)^{a+1}
\Bigl( \frac{\cos t_2}{\cos t_1}\Bigr)^{\frac12 a (a+1)}
\bigl(\cos^2 t_2\bigr)^{\frac12 a (a+1)}
\Bigl( \prod_{n=0}^a \binom{a}{n}\Bigr)^{-1} \\
\times 
\Bigl( \frac{1-p}{p}\Bigr)^{\frac12 a(a+1)}
(1-p)^{-\frac12 a(a+1)}
\end{gather*}
with $p= \frac{\cos^2 t_2}{\cos^2 t_2 - \cos^2 t_1}$ as
in Proposition \ref{prop:matrixentriesatinUnui} 
using that $\frac{(-a)_n}{(-1)^nn!}= \binom{a}{n}$.
Here we assume for the time being that $0<p<1$ so that 
all square roots are well-defined. Simplifying gives
\[
\det\bigl( Q_0(a_\bt)\bigr) 
= \pm \Bigl( \prod_{n=0}^a \binom{a}{n}\Bigr)^{-1}
\bigl(\cos t_1 \cos t_2\bigr)^{b(a+1)}
\bigl( \cos t_1 \cos t_2 (\cos^2 t_1 - \cos^2 t_2)\bigr)^{\frac12 a(a+1)}
\]
and this proves the statement for $0<p<1$. Since we know all 
entries of $S$ are polynomial in $(\cos t_1, \cos t_2)$, 
cf. Remark \ref{rmk:prop:matrixentriesatinUnui}(ii), 
the determinant of $S$ is polynomial in $(\cos t_1, \cos t_2)$ and the result holds for all $a_\bt$. 
\end{proof}

\begin{remark}\label{rmk:prop:detS}
Now by the results of Subsection \ref{ssec:generalsetup}
and \cite[Prop.~X.1.19]{Helg-1962} we have 
\eqref{eq:MatrixOPorthogonality} involving the matrix weight $W$, 
hence $S$. 
In this  case $\de\colon A \to \R$ is given by 
\begin{equation}\label{eq:deltaaexplicit}
\begin{split}
\de(a_\bt) &= (\sin t_1)^{2(m-2)}\, (\sin t_2)^{2(m-2)} \, \sin(2t_1)\, \sin(2t_2)\, \sin^2(t_1+t_2)\, \sin^2(t_1-t_2) \\
& = 4 (\sin t_1)^{2m-3}\, (\sin t_2)^{2m-3}\, 
\cos t_1\, \cos t_2\, ( \cos^2t_1-\cos^2t_2)^2, 
\end{split}
\end{equation}
see \cite[\S X.5]{Helg-1962} using Appendix \ref{ssec:appAstructurethy}. 
In particular, from Proposition \ref{prop:detS} and \eqref{eq:deltaaexplicit} we see that $\det(S(a_\bt))=0$ 
implies $\de(a_\bt)=0$. 
\end{remark}

\section{Radial part of the Casimir operator}

In order to obtain precise information on the matrix spherical
functions in its relation to the matrix functions $Q^\mu_{\nu_i}$
in Theorem \ref{thm:Qmunuisasmatrixspherfunction} 
and Corollary \ref{cor:thm:Qmunuisasmatrixspherfunction} 
we use the Casimir operator. Since the Casimir operator acts 
as a multiple of the identity in a representation $\pi^G_\la$ with scalar
$c_\la = \langle \la,\la\rangle + 2\langle \la, \rho\rangle$, where 
$\rho=\frac12 \sum_{\al\in \De^+}\al$, see \cite[Prop.~5.28]{Knap}, we have 
\begin{equation}\label{eq:PhimulaeigenfunctionCasimir}
R^\mu(\Om) \Phi^\mu_\la\vert_A = c_\la\, \Phi^\mu_\la\vert_A,
\end{equation}
where $R^\mu(\Om)$ is the radial part of the Casimir operator 
as in Appendix \ref{sec:AppRadial}. For convenience, the explicit expression 
of $R^\mu(\Om)$ is derived in Appendix \ref{sec:AppRadial}. 
The functions $\Phi^\mu_\la\vert_A$ are eigenfunctions of a much larger class of differential operators arising from a subalgebra 
of the universal enveloping algebra \cite[Ch.~9]{Dixm}, but we only use 
the Casimir operator. 
The eigenvalues play an important role in order to distinguish 
the eigenfunctions. 

\begin{lemma}\label{lem:eigvalueOmdifferent} 
Let $\la_1,\la_2\in P^+_G$ with $\la_1\prec \la_2$ and $\la_1\not=\la_2$, then $c_{\la_1}<c_{\la_2}$. 
\end{lemma}

\begin{proof} Rewrite $c_\la = \langle\la+\rho,\la+\rho\rangle 
- \langle \rho, \rho\rangle$, then 
\[
c_{\la_2}-c_{\la_1} = 
\langle\la_2+\rho,\la_2+\rho\rangle - \langle\la_1+\rho,\la_1+\rho\rangle = \langle \la_1+\la_2+2\rho, \la_2-\la_1\rangle
\]
and since $\la_1+\la_2+2\rho$ is in the interior of the positive
Weyl chamber and $\la_2-\la_1\in Q^+_G$, the right hand side is positive.
\end{proof}

As a first application we calculate the constants 
in Lemma \ref{lem:psi1intermsofphi1} 
and Lemma \ref{lem:explicitvKfixedvectorpsi2}. 

\begin{lemma}\label{lem:phi1phi2explicitpsi1psi2}
With the notation of Section \ref{sec:specialcases} we have
as functions on $A$
\begin{equation*}
\psi_1 =\frac{2m \phi_1+4}{m+2} , \qquad
\psi_2 = \frac{m-1}{m+2}\phi_2 +\frac{2(m+1)}{(m+2)^2}  \phi_1
+ \frac{2(m+1)(2m-1)}{m^2(m+2)^2}
\end{equation*}
\end{lemma}

\begin{remark}\label{rmk:lem:phi1phi2explicitpsi1psi2}
The relation is invertible; 
\begin{equation*}
\phi_1 =\frac{(m+2)\psi_1-4}{m}, 
\qquad 
\phi_2 = \frac{m(m+1)\psi_2 - (m+1)\psi_1+2}{m(m-1)}
\end{equation*}
\end{remark}

\begin{proof} In case $\mu=0$, $R^0(\Om)$ is an
explicit second order
partial differential operator, see \eqref{eq:AppCasimir4}, 
where all terms involving $\pi^K_\mu$ are set to zero. 
Put $f_1(t_1,t_2)=\psi_1(a_\bt) = \cos^2t_1+\cos^2t_2$, then 
by a trigonometric calculation (or using computer algebra) 
$R^0(\Om) f_1  = (2m+4) f_1 - 8$. Since  
\begin{equation*}
c_{\la_1} = \langle \om_1+\om_{m+1}, \sum_{n=1}^{m+1}\al_n
\rangle + 2 \hspace*{-.2truecm}\sum_{1\leq i<j\leq m+2} 
\langle \om_1+\om_{m+1}, \sum_{p=i}^{j-1}\al_p\rangle = 2m+4
\end{equation*}
using \eqref{eq:lambda12intermsofalphas}, we get that 
$-8a=c_{\la_1}b$ when writing $\phi_1=a\psi_1+b$ 
using Lemma \ref{lem:psi1intermsofphi1} and $R^0(\Om)\phi_1=c_{\la_1}\phi_1$. Evaluating
at the identity, using $\phi_1(e)=1$, $\psi_1(e)=2$, fixes
the constant. 

For $\psi_2$, put $f_2(t_1,t_2) =\psi(a_\bt)=\cos^2t_1 \, \cos^2t_2$. Then we find 
$R^0(\Om)f_2=(4m+4)f_2-2f_1$. In this case $c_{\la_2}=4m+4$, 
and we identify the expansion by considering the 
eigenvalue equation and the evaluation at $e$ using the first
result as well.
\end{proof}

Next we go back to the situation of $\mu=a\om_1+b\om_2$ 
with $a,b\in \N$. The basis $(F_1^k\cdot v_\mu)_{k=0}^a$,
$v_\mu$ being the highest weight vector of $V^K_\mu$, gives 
the $M$-decomposition, and  
\begin{equation}\label{eq:actionLiealgkforROm}
\begin{split}
\pi^K_\mu&(E_{1,1})F_1^k\cdot v_\mu = (a+b-k)F_1^k\cdot v_\mu, \qquad
\pi^K_\mu(E_{2,2}) F_1^k\cdot v_\mu = (b+k)F_1^k\cdot v_\mu, \\
&\pi^K_\mu(F_1)F_1^k\cdot v_\mu = F_1^{k+1}\cdot v_\mu, \qquad
\pi^K_\mu(E_1)F_1^k\cdot v_\mu = k(a-k+1)F_1^{k-1}\cdot v_\mu
\end{split}
\end{equation}
and we see that almost all actions of the Lie algebra in 
the expression for the radial part of the Casimir operator 
$R^\mu(\Om)$ of \eqref{eq:AppCasimir4} commute with the 
action of $M$. Only for the third line of the expression for 
$R^\mu_m(\Om)$ corresponding to the middle roots of $\mathrm{BC}_2$  we get a nontrivial interaction of the $M$-types. 
Put $G_k = \langle G(\cdot)F_1^k\cdot v_\mu, F_1^k\cdot v_\mu\rangle \colon A \to \C$ for the 
scalar action of $G\colon A \to \End_M(V^K_\mu)$ on $V^M_{\si_k}\subset V^K_\mu$, then we
can rewrite the radial part of the Casimir operator,
see Appendix \ref{ssec:appAleftoinvDOCasimirelt}, as 
\begin{equation}\label{eq:expRmuOm}
(R^\mu(\Om)G)_k = (R^\mu(\Om_\Lm)G)_k -
\frac12 \sum_{p=1}^2 \frac{\partial^2G_k}{\partial t_p^2} + 
\bigl(R^\mu_s(\Om)G\bigr)_k  + \bigl(R^\mu_m(\Om)G \bigr)_k +\bigl(R^\mu_l(\Om)G\bigr)_k,
\end{equation}
where the respective parts are given by
\begin{equation*}
(R^\mu(\Om_\Lm)G)_k = 
\frac{1}{2(m+2)} 
\bigl( m (a+b-k)^2 - 4(a+b-k)(b+k) + m (b+k)^2\bigr)G_k 
\end{equation*}
for the action corresponding to $\Om_\Lm$, and the term
for the short roots is equal to 
\begin{equation*}
\bigl(R^\mu_s(\Om)G\bigr)_k 
= -(m-2) 
\sum_{i=1}^2 \frac{\cos t_i}{\sin t_i} \frac{\partial G_k}{\partial t_i}
\end{equation*}
and the term for the middle roots gives 
\begin{equation*}
\begin{split}
\bigl(R^\mu_m(&\Om)G\bigr)_k 
=
- \frac{\cos(t_1+t_2)}{\sin(t_1+t_2)} \Bigl(
\frac{\partial G_k}{\partial t_1}+
\frac{\partial G_k}{\partial t_2}\Bigr) 
- \frac{\cos(t_1-t_2)}{\sin(t_1-t_2)} \Bigl(
\frac{\partial G_k}{\partial t_1}-
\frac{\partial G_k}{\partial t_2}\Bigr) \\
&- \Bigl(\frac{\cos(t_1+t_2)}{\sin^2(t_1+t_2)}
+ \frac{\cos(t_1-t_2)}{\sin^2(t_1-t_2)}\Bigr)
\bigl( (k+1)(a-k)G_{k+1} 
+  k(a-k+1) G_{k-1}  \bigr)  \\ & +
\Bigl( \frac{1}{\sin^2(t_1+t_2)} + 
\frac{1}{\sin^2(t_1-t_2)}\Bigr)
\bigl( ((k+1)(a-k)+k(a-k+1)) G_k  \bigr) 
\end{split}
\end{equation*}
and the term for the long roots simplifies to 
\begin{equation*}
\begin{split}
\bigl(R^\mu_l(&\Om)G\bigr)_k 
=
- \sum_{i=1}^2  
\frac{\cos(2t_i)}{\sin(2t_i)} 
\frac{\partial G_k}{\partial t_i} 
+ \frac{(a+b-k)^2}{2\cos^2t_1}G_k +
\frac{(b+k)^2}{2\cos^2t_2}G_k.
\end{split}
\end{equation*}

Having described the radial part of the Casimir operator
explicitly, we can use the action to make the constants 
in Theorem \ref{thm:Qmunuisasmatrixspherfunction} and 
Corollary \ref{cor:thm:Qmunuisasmatrixspherfunction}
explicit. 

\begin{prop}\label{prop:explicitQsinPhis} As
functions $A\to \End_M(V^K_\mu)$ we have 
\[
\Phi^\mu_{\nu_i} = \frac{(m+b+i)_i}{(m)_i}
\sum_{r=0}^i  
\frac{(-i)_{i-r}(-i-b)_{i-r}}{(i-r)!\, (1-m-2i-b)_{i-r}}\,
Q^\mu_{\nu_r}
\]
\end{prop}

The key ingredient in the proof of 
Proposition \ref{prop:explicitQsinPhis} is the action of the 
radial part of the Casimir operator on the functions 
$Q^\mu_{\nu_i}\colon A \to \End_M(V^K_\mu)$. 

\begin{lemma}\label{lem:acrionradpartonQs}
For $i\in\{0,\ldots,a\}$ and $Q^\mu_{\nu_i}$ as in Theorem \ref{thm:Qmunuisasmatrixspherfunction} we have
\[
R^\mu(\Om) Q^\mu_{\nu_i} = c_{\nu_i} Q^\mu_{\nu_i} 
- 2i(b+i) Q^\mu_{\nu_{i-1}}
\]
where $c_{\nu_i}$ is the eigenvalue of $\Phi^\mu_{\nu_i}$
for $R^\mu(\Om)$;
\[
c_{\nu_i} = 2i^2+2i(b+m) +
(m+1)a+2mb + \frac{1}{m+2}\bigl( (m+1)a^2 +2mb(a+b)\bigr).
\]
\end{lemma}

\begin{proof} Note that $R^\mu(\Om)\Phi^\mu_{\nu_i} = c_{\nu_i}\Phi^\mu_{\nu_i}$ with $c_{\nu_i}=\langle \nu_i,\nu_i\rangle 
+ 2\langle \nu_i,\rho\rangle$, and the explicit value of $c_{\nu_i}$ 
follows by a calculation. This shows that $c_{\nu_{i}}<c_{\nu_{i+1}}$. 
This also follows more generally from 
Corollary \ref{cor:prop:cond3issatisfied}
and Lemma \ref{lem:eigvalueOmdifferent}. Since the transition
of the basis of $(\Phi^\mu_{\nu_i})_{i=0}^a$ to the 
basis $(Q^\mu_{\nu_i})_{i=0}^a$ is triangular, we find 
$R^\mu(\Om) Q^\mu_{\nu_i} = c_{\nu_i} Q^\mu_{\nu_i} 
+ \sum_{r=0}^{i-1} C_r Q^\mu_{\nu_{r}}$
for certain constants $C_r$. These constants can be determined
considering the action on $V^M_{\si_0}$ of this identity using $q^\mu_{\nu_i,\si_0}(a_\bt)
=(\cos t_1)^{a+b}(\cos t_2)^{b+2i}$ and 
\begin{equation}\label{eq:qforrho1}
q^\mu_{\nu_i,\si_1}(a_\bt) =\frac{a-i}{a} (\cos t_1)^{a+b-1}(\cos t_2)^{b+2i+1} +\frac{i}{a}(\cos t_1)^{a+b+1}(\cos t_2)^{b+2i-1}, 
\end{equation}
where we use $K_1(x;p,N)=1-\frac{x}{pN}$ for the Krawtchouk 
polynomials, see Theorem \ref{thm:Qmunuisasmatrixspherfunction}. Using this we find by a trigonometric calculation (using computer
algebra) 
\begin{gather*}
(R^\mu(\Om) Q^\mu_{\nu_i})_0 = 
c_{\nu_i}(\cos t_1)^{a+b}(\cos t_2)^{b+2i} 
- 2i(b+i) (\cos t_1)^{a+b}(\cos t_2)^{b+2i-2}.
\end{gather*}
The right hand side is $c_{\nu_i}
q^\mu_{\nu_i,\si_0}(a_\bt) - 2i(b+i)
q^\mu_{\nu_{i-1},\si_0}(a_\bt)$, so that 
$C_{i-1}=- 2i(b+i)$ and $C_r=0$ for $r<i-1$ since the 
functions $q^\mu_{\nu_i,\si_k}$ are independent for $i\in\{0,\ldots, a\}$. 
\end{proof}

\begin{remark}\label{rmk:lem:acrionradpartonQs}
The fact that the right hand side of Lemma \ref{lem:acrionradpartonQs} consists of just two 
matrix leading terms makes it possible to derive 
many explicit results for the matrix spherical functions. 
This is one of the main motivations to consider these specific 
leading terms. 
\end{remark}

\begin{proof}[Proof of Proposition \ref{prop:explicitQsinPhis}]
Apply $R^\mu(\Om)$ to Corollary \ref{cor:thm:Qmunuisasmatrixspherfunction}, using that   
the $\Phi^\mu_{\nu_i}$ are eigenfunctions for $R^\mu(\Om)$, 
Lemma \ref{lem:acrionradpartonQs} and that the $Q^\mu_{\nu_i}$ 
are linearly independent to find the recursion
$d^i_r \, c_{\nu_i} = d^i_r \, c_{\nu_r}- d^i_{r+1} 2(r+1)(b+r+1)$ for $r<i$. Using the value for $c_{\nu_i}$ as 
in Lemma \ref{lem:acrionradpartonQs} we obtain 
\begin{equation*}
d^i_r(i-r)(b+m+r+i)=-d^i_{r+1}(r+1)(b+r+1) \ 
\Longrightarrow \  
d^i_r= \frac{(-i)_{i-r}(-i-b)_{i-r}}{(i-r)!\, (1-m-2i-b)_{i-r}}
d^i_i
\end{equation*}
by iteration 
and it remains to determine $d^i_i$. Evaluating at the 
identity element 
$e\in A$ and using that $Q^\mu_{\nu_r}$ and $\Phi^\mu_{\nu_i}$
are the identity at $e$, we find
\[
\frac{1}{d^i_i} =
\sum_{r=0}^i  
\frac{(-i)_{i-r}(-i-b)_{i-r}}{(i-r)!\, (1-m-2i-b)_{i-r}} 
= \rFs{2}{1}{-i,-i-b}{1-m-2i-b}{1}
=  \frac{(1-m-i)_i}{(1-m-2i-b)_i}
\]
by the Chu-Vandermonde summation, 
see e.g. \cite[Cor.~2.2.3]{AndrAR},  \cite[\S 1.4]{Isma}. 
Simplifying $d^i_i$ gives the result.
\end{proof}

As a next step we translate the Proposition \ref{prop:explicitQsinPhis} into the transition for the 
matrix weight $W$ and $S$. Recall the matrix functions
$\Phi_0$ and $Q_0$ as defined in  \S \ref{sec:MatrixWeight}. 

\begin{prop}\label{prop:Phi0isLQ0}
We have $\Phi_0 = L Q_0$ with the constant lower triangular matrix $L$ given by $L_{i,j}=0$ for $j>i$ and 
\begin{gather*}
L_{i,j} = 
(-1)^{i+j} \binom{i}{j} 
\frac{(m+b+i)_i}{(m)_i} \frac{(b+j+1)_{i-j}}{(m+i+j+b)_{i-j}},
\qquad 0\leq j\leq i \leq a 
\end{gather*}
and its inverse is the lower triangular matrix
given by $(L^{-1})_{i,j}=0$ for $j>i$ and
\begin{gather*}
(L^{-1})_{i,j} = 
 \binom{i}{j} 
\frac{(m)_j}{(m+b+j)_j} \frac{(b+j+1)_{i-j}}{(m+2j+b-1)_{i-j}},
\qquad 0\leq j\leq i \leq a . 
\end{gather*}
\end{prop}

\begin{proof} Recall from Section \ref{sec:MatrixWeight}
and Proposition \ref{prop:explicitQsinPhis}
that as functions on $A$ we have
\begin{gather*}
(\Phi_0)_{i,k} = \phi^\mu_{\nu_i,\si_k} = 
(\Phi^\mu_{\nu_i})_k = 
\sum_{r=0}^i d^i_r \, (Q^\mu_{\nu_r})_k
= \sum_{r=0}^i d^i_r \, q^\mu_{\nu_r,\si_k} 
= \sum_{r=0}^a L_{i,r} \, (Q_0)_{r,k}
\end{gather*}
with $L_{i,r}=d^i_r$ for $i\leq r$ and 
$L_{i,r}=0$ for $i> r$. Rewriting gives the matrix $L$. 

To show that $L^{-1}$ is as given we need to show the 
nontrivial case; for $j\leq i$ we have to show 
$\sum_{r=j}^i L_{i,r} (L^{-1})_{r,j} =\de_{i,j}$. 
Taking out the $r$-independent parts, we see that this 
equivalent to showing
\begin{gather*}
\de_{i,j}= \sum_{r=j}^i 
(-1)^{i+r} \binom{i}{r} 
\frac{(b+r+1)_{i-r}}{(m+i+r+b)_{i-r}}
\binom{r}{j} 
\frac{(b+j+1)_{r-j}}{(m+2j+b-1)_{r-j}}.
\end{gather*}
The right hand side can be rewritten as 
\begin{gather*}
\frac{(b+i+j)_{i-j}}{(m+i+j+b-1)_{i-j}} \binom{i}{j} (-1)^{i+j}
\sum_{k=0}^{i-j}
\frac{ (j-i)_k (m+i+j+b-1)_k}{k!\, (m+2j+b)_k}
\end{gather*}
and the sum is a terminating ${}_2F_1$-series at $1$, which can 
be evaluated by the Chu-Van\-der\-monde summation, 
see e.g. \cite[Cor.~2.2.3]{AndrAR},  \cite[\S 1.4]{Isma}, 
as $\frac{(1+j-i)_{i-j}}{(m+2j+b)_{i-j}}$ so that the numerator
gives $0$ unless $i=j$, in which case we find $1$.
\end{proof}

\section{Matrix orthogonal polynomials in 
two variables}\label{sec:MVOP}

In Section \ref{sec:MatrixWeight} we have 
established the matrix weight for the 
polynomials, and in this section we establish 
some more properties for these matrix orthogonal
polynomials in two variables with $\mathrm{BC}_2$-symmetry. 
In particular, we make the orthogonality relations
more explicit. Moreover, we derive the 
matrix partial differential operator to which these 
matrix polynomials are eigenfunction.

Firstly, the Haar measure on $A$ is $dt_1dt_2$ on 
$[-\pi,\pi]\times [-\pi,\pi]$ and 
using the invariance under the sign changes, we can reduce to 
the integral over $[0,\frac12\pi]\times [0,\frac12\pi]$. 
Using \eqref{eq:deltaaexplicit} we 
find for the normalising constant in
\eqref{eq:MatrixOPorthogonality} 
\begin{equation}
\frac{1}{c} = \int_A |\de(a)|\, da = 4^2 \int_0^{\frac12\pi} \int_0^{\frac12\pi} 
|\de(a_\bt)|\, dt_1dt_2 = \frac{32}{m^2(m^2-1)}.
\end{equation}

\begin{figure}[ht]
\begin{tikzpicture}[scale=1]
    \draw[->,>=stealth](-3,0)--(3,0) node[below=1pt]{$x_1$};
    \draw[->,>=stealth](0,-2)--(0,2) node[right=1pt]{$x_2$};
    \coordinate (0,0) node [left=5pt,below=1pt]{0};
\foreach \y in {1,2}
    \node[below] at(\y,0){\y};
\foreach \y in {-2,-1}
    \node[left] at(0,\y){\y};
\foreach \y in {1,2}
    \node[left] at(0,\y){\y};
\foreach \y in {-2,-1}
    \node[below] at(\y,0){\y};
\filldraw[fill=gray!50, opacity=0.5] (0,0) -- plot [domain=0:2,smooth](\x,\x*\x/4) -- (2,1)-- (0,-1)-- (-2,1)-- plot [domain=-2:0,smooth] (\x,\x*\x/4) --(0,0);
\end{tikzpicture}
\caption{Integration region $I$ of Theorem \ref{thm:MVOPBC2}.}
\label{fig:Integrationregion}
\end{figure}
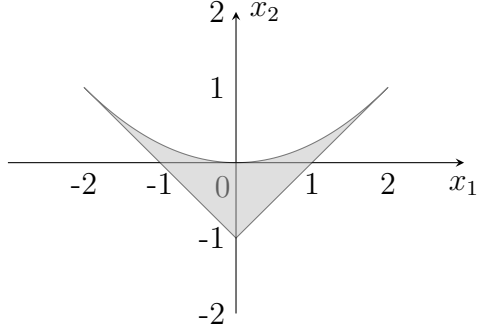 

In order to make the connection to the 
$\mathrm{BC}_2$-case as originally introduced by 
Koornwinder \cite{Koor-IndagM1},
\cite{Koor-Wisconsin}, see also \cite{Spri}, \cite{Opda}, we make an affine change of variable 
$\psi_1=\frac12 x_1+1$, $\psi_2 = \frac14 x_2+\frac14 x_1 +\frac14$, or,  in terms of $t_1$ and $t_2$, $x_1=\cos(2t_1)+\cos(2t_2)$,
$x_2= \cos(2t_1)\, \cos(2t_2)$. Then the map sending $(t_1,t_2)\in [0,\frac12\pi]\times 
[0,\frac12\pi]$ to $(x_1,x_2)$ is a $2:1$-mapping onto the region 
bounded by the parabola $x_1^2=4x_2$ and the lines $x_2=x_1-1$, 
$x_2=1-x_1$, see Figure 
\ref{fig:Integrationregion}. 
This is exactly the region of integration for the polynomials studied in 
\cite{Koor-IndagM1}, \cite{Koor-Wisconsin}, \cite{Spri}. 
For $\bd=(d_1,d_2)\in \N^2$ we define matrix polynomials 
$R_\bd$ of size $(a+1)\times(a+1)$ of degree $\bd$ by 
\begin{equation}\label{eq:defRdpols}
R_\bd(x_1,x_2) =  P_\bd(\phi_1,\phi_2)L
\end{equation}
where we use the notation for $P_\bd$ as in 
Subsection \ref{ssec:generalsetup}, the affine transformation from $(x_1,x_2)$ to $(\psi_1,\psi_2)$ as given 
above and the affine transformation from Lemma
\ref{lem:phi1phi2explicitpsi1psi2}, and with $L$ as  in 
Proposition \ref{prop:Phi0isLQ0}. Finally, we define the 
matrix weight
\begin{equation}
S^a(x_1,x_2) =  
S(\psi_1,\psi_2)
\end{equation}
with $S(\psi_1(a_\bt),\psi_2(a_\bt))$ defined in 
\eqref{eq:defSij} for the case $a\in\N$, $b=0$, and using 
the coordinate change of $(\psi_1,\psi_2)$ to $(x_1,x_2)$. In case 
$a=1$ we obtain 
\begin{equation*} 
S^1(x_1,x_2) = 
\begin{pmatrix}
\frac12x_1+1 & \frac12(x_1+x_2+1) \\ \frac12(x_1+x_2+1) & 
\frac14(\frac12x_1+1)(x_1+x_2+1) 
\end{pmatrix},
\end{equation*}
and for $a=2$ we obtain that $S^2(x_1,x_2)$ equals 
\begin{equation*}
\begin{pmatrix}
      (\frac{1}{2}x_{1}+1)^{2}-\frac{x_{1}+x_{2}+1}{4} & \frac{3}{8}(\frac{1}{2}x_{1}+1)(x_{1}+x_{2}+1) & \frac{3}{16}(x_{1}+x_{2}+1)^{2} \\
      \frac{3}{8}(\frac{1}{2}x_{1}+1)(x_{1}+x_{2}+1) & \frac{(x_{1}+x_{2}+1)(2(x_{1}+x_{2}+1)+(\frac{1}{2}x_{1}+1)^2)}{16} & \frac{3}{32}(\frac{1}{2}x_{1}+1)(x_{1}+x_{2}+1)^{2} \\
      \frac{3}{16}(x_{1}+x_{2}+1)^{2} & \frac{3}{32}(\frac{1}{2}x_{1}+1)(x_{1}+x_{2}+1)^{2} & \frac{(x_{1}+x_{2}+1)^{2}((\frac{1}{2}x_{1}+1)^{2}-\frac{x_{1}+x_{2}+1}{4})}{16} 
\end{pmatrix}.
\end{equation*} 
These examples follow from \eqref{eq:Sfora=1} taking $b=0$ and 
$\psi_1=\frac12 x_1+1$, $\psi_2 = \frac14 x_2+\frac14 x_1 +\frac14$. 

\begin{theorem}\label{thm:MVOPBC2} The matrix polynomials
$R_\bd$ defined by \eqref{eq:defRdpols} are orthogonal 
on the region of integration $I$ as in Figure 
\ref{fig:Integrationregion}.
and 
\begin{multline*}
\iint_I R_\bd(x_1,x_2) S^a(x_1,x_2) 
\bigl( R_{\bd'}(x_1,x_2)\bigr)^\ast 
(1-x_1+x_2)^{m-2}(1+x_1+x_2)^b (x_1^2-4x_2)^{\frac12}\, dx_1dx_2 \\ 
= \de_{\bd,\bd'} 2^{2m+2b-10} m^2(m^2-1) H_\bd
\end{multline*}
where $S^a$ is positive definite on $I$ with positive determinant
on the interior of $I$. Moreover, the weight function is 
indecomposable. Here the matrix $H_\bd$ is a diagonal matrix
with $(H_\bd)_{k,k}= (a+1)^2/\dim V^G_{\nu_k+d_1\la_1+d_2\la_2}$.

Moreover, the polynomials $R_\bd$ are eigenfunctions to a second
order matrix partial differential operator;
\begin{gather*}
R_\bd R^0(\Om) - R_\bd C^\mu
+ R_\bd (\Lambda_0+S) = \Lambda_\bd R_\bd  
\end{gather*}
where $\Lambda_\bd=\diag(c_{\nu_i+d_1\la_1+d_2\la_2})_{i=0}^a$,
$\bd=(d_1,d_2)\in \N^2$, $\Lambda_0=\Lambda_{(0,0)}$, and 
$S$ is the lower triangular matrix with one nonzero subdiagonal
with $S_{r,r-1}=-2r(b+r)$. The operator $R^0(\Om)$ is the second order
partial differential operator acting from the right as the 
identity times the classical partial differential 
operator 
\begin{gather*}
(2x_1^2-4x_2-4)\frac{\partial^2}{\partial x_1^2}
+(-2x_1^2+4x_2^2+4x_2)\frac{\partial^2}{\partial x_2^2}
+4x_1(x_2-1) \frac{\partial^2}{\partial x_1\partial x_2} \\
+ 2 \bigl( (m+2)x_1 +2m-4\bigr) \frac{\partial}{\partial x_1}
+ 2\bigl( (m-2)x_1+2 +(2m+2)x_2\bigr) \frac{\partial}{\partial x_2}
\end{gather*}
and $C^\mu$ is the first order matrix differential operator 
$\frac{\partial}{\partial x_1}C^\mu_1
+\frac{\partial}{\partial x_2}C^\mu_2$
where $C_1^\mu$ and $C_2^\mu$ are tridiagonal polynomial matrices of degree $1$ given by
\begin{gather*}
(C^\mu_1)_{r,r} =2\bigl( (a+b+r)x_1-2b-2r\bigr), \ 
(C^\mu_2)_{r,r} =2\bigl( (b+r)(2x_2-x_1)+ a(x_2+1))\bigr), \\
(C^\mu_1)_{r,r-1} = (C^\mu_2)_{r,r-1} = -r(x_1+x_2+1),\qquad
(C^\mu_1)_{r,r+1} = (C^\mu_2)_{r,r+1} = -4(a-r).
\end{gather*}
Moreover, the matrix partial differential operator is symmetric
with respect to the matrix weight. 
\end{theorem}

\begin{remark}\label{rmk:thm:MVOPBC2}
Note that the scalar part of the weight in Theorem \ref{thm:MVOPBC2} 
is the weight considered by Koornwinder \cite{Koor-IndagM1}, 
\cite{Koor-Wisconsin} and Sprinkhuizen-Kuyper \cite{Spri} for the 
special case $\al=m-2$, $\be=b$, $\ga=\frac12$. 
Similarly, in the case $\mu=0$, i.e. $a=b=0$, the partial 
differential operator reduces to the 
partial differential operator studied in \cite{Koor-IndagM1}, 
\cite{Koor-Wisconsin}, \cite{Spri} up to a scalar multiple 
for these choices of parameters. 
The case $a=0$, $b\in \N$, gives the case 
of a non-trivial character of $K$, and this
corresponds to Heckman \cite[Ch.~5]{HeckS}.

Note that in the scalar
case the $2$-variable orthogonal polynomials can be expressed in terms
of Jacobi polynomials \cite[$(3.13)$]{Koor-Wisconsin}, 
\cite[Lemma~3.1]{Spri}. 
It is not clear if in this case we also have 
an explicit expression of $R_\bd(x_1,x_2)$ in terms of 
matrix Jacobi polynomials of a single variable. 
\end{remark}

\begin{proof}[Proof of the orthogonality in Theorem \ref{thm:MVOPBC2}]
Observe that the Jacobian for the change of $(t_1,t_2)$ to 
$(x_1,x_2)$ is given by
\[
32\, |\sin(t_1)\sin(t_2)\cos(t_1)\cos(t_2)\bigl( \cos^2(t_1)-\cos^2(t_2\bigr)|
\]
and $\sin^2(t_1)\sin^2(t_2)=\frac14(1-x_1+x_2)$,
$(\cos^2(t_1)-\cos^2(t_2))^2=\frac14(x_1^2-4x_2)$. Keeping track
of the constants involved, the statements on the orthogonality 
follow from \eqref{eq:MatrixOPorthogonality} and from Section \ref{sec:MatrixWeight}, in particular 
Proposition \ref{prop:detS} and Remark \ref{rmk:prop:detS}, 
and $S=Q_0Q_0^\ast$ being positive. 
\end{proof}

In order to prove the statement of 
Theorem \ref{thm:MVOPBC2} concerning the partial differential operator, we need to be able to rewrite 
the eigenvalue equation of the radial part of the Casimir
operator $R^\mu(\Om)$ acting on the eigenvector $\Phi^\mu_\la\vert_A$
in terms of an operator acting on the polynomials 
$R_\bd$. For this we need to conjugate $R^\mu(\Om)$ 
with the matrix function $Q_0$, see \cite[\S 3.2]{KoelvPR-JFA}. 
We collect the technical result in Lemma \ref{lem:dQdpsi}.

\begin{lemma}\label{lem:dQdpsi}
We have for $i=1,2$ as matrix valued functions on $A$ 
\begin{gather*}
\frac{\partial \psi_i}{\partial t_1}\frac{\partial Q_0}{\partial t_1}
+\frac{\partial \psi_i}{\partial t_2}\frac{\partial Q_0}{\partial t_2} = C_i(\psi_1,\psi_2) Q_0
\end{gather*}
where we consider the functions as functions of $(t_1,t_2)$ by 
evaluating at $a_\bt\in A$. Here  
$C_i(\psi_1,\psi_2)$ is a matrix polynomial in $(\psi_1,\psi_2)$ of total degree at most $1$, where the non-zero entries are explicitly given by 
\begin{gather*}
C_1(\psi_1,\psi_2)_{r,r} = 2(a+2b+2r-(a+b+r)\psi_1), \\ 
C_2(\psi_1,\psi_2)_{r,r}  
= 2((b+r)\psi_1 -(a+2b+2r)\psi_2), 
\\ 
C_1(\psi_1,\psi_2)_{r,r-1} = C_2(\psi_1,\psi_2)_{r,r-1} = 2r \psi_2, \qquad 
C_1(\psi_1,\psi_2)_{r,r+1} = C_2(\psi_1,\psi_2)_{r,r+1} = 2(a-r).
\end{gather*}
\end{lemma}

Note that the tridiagonal matrices coincide on the off-diagonal entries. 

There are analogues of Lemma \ref{lem:dQdpsi} with $Q_0$ replaced
by $\Phi_0$ and $\psi_i$ replaced by $\phi_i$ or $x_i$, see also the 
first paragraph of the proof. However, in general it is hard to calculate the right hand side explicitly. In this case we can do the explicit calculation because of the homogeneity properties of 
the entries of $Q_0$ and $\psi_1$, $\psi_2$. 

\begin{proof} Lemma 3.9 of \cite{KoelvPR-JFA} implies that 
\begin{gather*}
\frac{\partial \phi_i}{\partial t_1}\frac{\partial \Phi_0}{\partial t_1}
+\frac{\partial \phi_i}{\partial t_2}\frac{\partial \Phi_0}{\partial t_2} = C'_i(\phi_1,\phi_2) \Phi_0
\end{gather*}
for a matrix polynomial $C'_i$ in $(\phi_1,\phi_2)$ of maximal total degree $1$, where we use the adjoint of \cite[Lemma 3.9]{KoelvPR-JFA}. Using $\Phi_0=LQ_0$ and the affine transformation of 
$(\phi_1,\phi_2)$ to $(\psi_1,\psi_2)$ given in 
Lemma \ref{lem:phi1phi2explicitpsi1psi2} proves the general statement of the lemma, and it remains to determine the polynomials $C_i$. 

Take $i=1$ and consider the $(r,s)$-entry of the left hand side of the identity. Since 
$(Q_0)_{r,s}(a_\bt)= q^\mu_{\nu_r,\si_s}(a_\bt)$ 
is a  homogeneous polynomial of degree $a+2b+2r$ in 
$(\cos t_1, \cos t_2)$, we see by an explicit calculation that 
\begin{equation}\label{eq:lem:dQdpsi-1}
\frac{\partial \psi_1}{\partial t_1}\frac{\partial (Q_0)_{r,s}}{\partial t_1}
+\frac{\partial \psi_1}{\partial t_2}\frac{\partial (Q_0)_{r,s}}{\partial t_2} = 2(a+2b+2r) (Q_0)_{r,s} + \cE_{r,s}
\end{equation}
where $\cE_{r,s}$ is a homogeneous polynomial of degree 
$a+2b+2r+2$ in $(\cos t_1,\cos t_2)$. Since $\psi_1$, respectively $\psi_2$, is homogeneous 
of degree $2$, respectively $4$, and the fact that $C_1(\psi_1,\psi_2)$ is of degree at most $1$, we have 
$\cE_{r,s} = a_r \psi_1 (Q_0)_{r,s} +b_r \psi_2 (Q_0)_{r-1,s} + 
c_r (Q_0)_{r+1,s}$ for coefficients $a_r$, $b_r$ and $c_r$. 
So we see that $C_1(\psi_1,\psi_2)$ is a tridiagonal matrix, and we have to determine the coefficients. For this we take $s=0$ and 
recall from Theorem \ref{thm:Qmunuisasmatrixspherfunction} 
that $(Q_0)_{r,0}(a_\bt) = q^\mu_{\nu_r,\si_0}(a_\bt) = (\cos t_1)^{a+b}(\cos t_2)^{b+2r}$. So that 
\begin{multline*}
\frac{\partial \psi_1}{\partial t_1}\frac{\partial (Q_0)_{r,0}}{\partial t_1}
+\frac{\partial \psi_1}{\partial t_2}\frac{\partial (Q_0)_{r,0}}{\partial t_2} = 2(a+b+2r) (Q_0)_{r,0} \\ 
-2 (a+b)  (\cos t_1)^{a+b+2}(\cos t_2)^{b+2r}
-2 (b+2r) (\cos t_1)^{a+b}(\cos t_2)^{b+2r+2}
\end{multline*}
and comparing with the explicit form of $\cE_{r,s}$ we get 
$a_r+b_r=-2(a+b)$ and $a_r+c_r=-2(b+2r)$. Writing $b_r$ and 
$c_r$ in terms of $a_r$, we now take $s=1$ in 
\eqref{eq:lem:dQdpsi-1} and we use
the explicit expression \eqref{eq:qforrho1} in order to 
obtain by a calculation (using computer algebra) that 
$a_r= - 2a-2b-2r$. This gives the expression for 
$C_1(\psi_1,\psi_2)$.

In case $i=2$ we proceed similarly and we get 
\begin{equation} \label{eq:lem:dQdpsi-2}
\frac{\partial \psi_2}{\partial t_1}\frac{\partial (Q_0)_{r,s}}{\partial t_1}
+\frac{\partial \psi_2}{\partial t_2}\frac{\partial (Q_0)_{r,s}}{\partial t_2} = - 2(a+2b+2r) \psi_2 (Q_0)_{r,s} + \cE_{r,s}
\end{equation}
where again $\cE_{r,s}$ is a homogeneous polynomial of degree 
$a+2b+2r+2$ in $(\cos t_1,\cos t_2)$ and hence of the form 
$\cE_{r,s} = a_r \psi_1 (Q_0)_{r,s} +b_r \psi_2 (Q_0)_{r-1,s} + 
c_r (Q_0)_{r+1,s}$ as before. So also $C_2(\psi_1,\psi_2)$ is tridiagonal. Taking $s=0$ in \eqref{eq:lem:dQdpsi-2}  we find by a calculation that 
$a_r+b_r=2(b+2r)$ and $a_r+c_r=2(a+b)$ in this case. Eliminating
$b_r$ and $c_r$ in terms of $a_r$ and now taking $s=1$ in \eqref{eq:lem:dQdpsi-2} and 
using the explicit form \eqref{eq:qforrho1}, we find by a calculation (using computer algebra) $a_r=2b+2r$. 
This gives the expression for 
$C_2(\psi_1,\psi_2)$.
\end{proof}

In order to derive the partial differential operator of
Theorem \ref{thm:MVOPBC2}, we observe that in this case we can rewrite 
\eqref{eq:expPhumulaintoPhibottom} as 
\begin{equation}\label{eq:expPhumulaintoQnu}
\Phi^\mu_{\nu_i+\la_\sph}(a_\bt)  
= \sum_{r=0}^a q^\mu_{\nu_i,v_r;\bd}(\psi_1(a_\bt),\psi_2(a_\bt))
\, Q^\mu_{\nu_r}(a_\bt)
\end{equation}
using Theorem \ref{thm:Qmunuisasmatrixspherfunction} and 
Remark \ref{rmk:lem:phi1phi2explicitpsi1psi2}. Note that 
$q^\mu_{\nu_i,v_r;\bd}$ is a polynomial of total degree $|\bd|$. 
Note that $q^\mu_{\nu_i,v_r;\bd}$ are entries of 
the matrix polynomials $R_\bd$ up to a change of coordinates. 
Since $\Phi^\mu_{\nu_i+\la_\sph}(a_\bt)$ is an eigenvector of 
the radial part $R^\mu(\Om)$ of the Casimir operator, we need
to derive the action $R^\mu(\Om)$ on $\bt \mapsto f(\psi_1(a_\bt),\psi(a_\bt)) Q^\mu_{\nu_r}(a_\bt)$ for $f$ a $2$-variable scalar function. It can be checked from 
\eqref{eq:expRmuOm} that, cf. proof of 
\cite[Lemma 3.9]{KoelvPR-JFA},
\begin{equation}\label{eq:RadpartLeibniz}
R^\mu(\Om)\bigl( f(\psi_1,\psi_2)Q^\mu_{\nu_r}) 
= f(\psi_1,\psi_2) \bigl(R^\mu(\Om) Q^\mu_{\nu_r}\bigr)
+ \bigl(R^0(\Om)f(\psi_1,\psi_2) \bigr)
Q^\mu_{\nu_r} 
- \sum_{p=1}^2 \frac{\partial f}{\partial t_p}
\frac{\partial Q^\mu_{\nu_r}}{\partial t_p}
\end{equation}
and the first term follows from Lemma \ref{lem:acrionradpartonQs} 
and the last term can be dealt with using 
Lemma \ref{lem:dQdpsi} and the chain rule. So we can rewrite 
$R^\mu(\Om)\bigl( f(\psi_1,\psi_2)Q^\mu_{\nu_r})$ completely
in terms of $Q^\mu_{\nu_s}$'s. 
So we get an eigenvalue equation for the $q^\mu_{\nu_i,\nu_r; \bd}$, which is  
\begin{gather}\label{eq:eigvaleqtforqs}
c_{\nu_i+\la_\sph} \sum_{r=0}^a q^\mu_{\nu_i,\nu_r; \bd}(\psi_1,\psi_2) Q^\mu_{\nu_r}  = 
\sum_{r=0}^a \bigl( R^0(\Om)q^\mu_{\nu_i,\nu_r; \bd}(\psi_1,\psi_2)\bigr) Q^\mu_{\nu_r} 
+ \\ \nonumber  \sum_{r=0}^a q^\mu_{\nu_i,\nu_r; \bd}(\psi_1,\psi_2)\bigl( c_{\nu_r}Q^\mu_{\nu_r} - 2r(b+r)Q^\mu_{\nu_{r-1}}\bigr) 
- \sum_{p=1}^2 \sum_{r,u=0}^a\frac{\partial q^\mu_{\nu_i,\nu_r; \bd}}{\partial \psi_p} (\psi_1,\psi_2)  C_p(\psi_1,\psi_2)_{r,u} Q^\mu_{\nu_u}
\end{gather}
where $\la_\sph=d_1\la_1+d_2\la_2$, $\bd=(d_1,d_2)\in \N^2$.  

\begin{lemma}\label{lem:PDEforQ}
Let $Q_\bd= Q_\bd(\psi_1,\psi_2)$ be the matrix polynomial defined by $(Q_\bd)_{i,j}(\psi_1,\psi_2) = q^\mu_{\nu_i,\nu_j; \bd}(\psi_1,\psi_2)$ using \eqref{eq:expPhumulaintoQnu}, then 
\begin{gather*}
Q_\bd R^0(\Om) - \frac{\partial Q_\bd}{\partial\psi_1}C_1(\psi_1,\psi_2) - \frac{\partial Q_\bd}{\partial\psi_2}C_2(\psi_1,\psi_2) 
+ Q_\bd (\Lambda_0+S) = \Lambda_\bd Q_\bd  
\end{gather*}
where $\Lambda_\bd$, $\Lambda_0$, and 
$S$ are as in Theorem \ref{thm:MVOPBC2}. $C_i(\psi_1,\psi_2)$, $i=1,2$, are the matrix polynomials of at most degree $1$, see Lemma \ref{lem:dQdpsi}. Moreover, $R^0(\Om)$ is a matrix second order partial differential operator in $(\psi_1,\psi_2)$ acting entrywise, considered as acting from the right. 
\end{lemma}

\begin{remark}\label{rmk:lem:PDEforQ}
Note that the radial part $R^0(\Om)$ acts as a matrix differential operator when considered as multiplied by the identity. This has to be rewritten as differential operator with respect to the variables 
$(\psi_1,\psi_2)$, which can be done since the spherical functions are polynomials in $(\phi_1,\phi_2)$, hence in $(\psi_1,\psi_2)$,
see Vretare \cite{Vret}. 
For convenience we write down the terms of 
$R^0(\Om)f$, where $f$ is a scalar polynomial in $(\psi_1,\psi_2)$.
Then $R^0(\Om_\Lm)$ is zero, and 
 $-\frac12 \sum_{p=1}^2 \frac{\partial^2}{\partial t_p^2}$ in 
\eqref{eq:expRmuOm} corresponds to 
\begin{equation*}
\begin{split}
&(2\psi_1-2)\frac{\partial f}{\partial\psi_1} +
(4\psi_2-\psi_1)\frac{\partial f}{\partial\psi_2}+
(2\psi_1^2-2\psi_1-4\psi_2)\frac{\partial^2 f}{\partial\psi_1^2}
 + \\
&\qquad (4\psi_2^2-2\psi_1\psi_2)\frac{\partial^2 f}{\partial\psi_2^2} + 
(4\psi_1\psi_2-8\psi_2)\frac{\partial^2 f}{\partial\psi_1\partial\psi_2}.
\end{split}
\end{equation*}
$R^0_s(\Om)f$  corresponds to 
$2(m-2)\psi_1 \frac{\partial f}{\partial\psi_1} 
+4(m-2) \psi_2 \frac{\partial f}{\partial\psi_2}$ and $R^0_l(\Om)f$ corresponds to 
$(2\psi_1-2) \frac{\partial f}{\partial\psi_1} 
+(4\psi_2-\psi_1) \frac{\partial f}{\partial\psi_2}$ and 
$R^0_m(\Om)f$ corresponds 
to 
$(4\psi_1-4) \frac{\partial f}{\partial\psi_1} 
+4\psi_2 \frac{\partial f}{\partial\psi_2}$. 
\end{remark}

\begin{proof}
Write \eqref{eq:eigvaleqtforqs} in matrix notation, then we obtain 
the result of the lemma multiplied by the matrix function $Q_0$ from the right. Since $Q_0$ is generically invertible, see 
the proof of 
Proposition \ref{prop:detS}, the lemma follows. 
\end{proof}

\begin{proof}[Proof of the partial differential equation 
in Theorem \ref{thm:MVOPBC2}]
Comparing \eqref{eq:expPhumulaintoQnu} with 
\eqref{eq:defRdpols} and \eqref{eq:expPhumulaintoPhibottom}, we
see that $R_\bd$ and $Q_\bd$ are the same up to the change of
coordinates from $(\psi_1,\psi_2)$ for $Q_\bd$ to 
$(x_1,x_2)$ for $R_\bd$. Note that $x_1=2\psi_1-2$, $x_2=4\psi_1-2\psi_1+1$, making this affine change 
of coordinates  gives  the expression of $R^0(\Om)$ in the 
$(x_1,x_2)$-coordinates as given in Theorem \ref{thm:MVOPBC2}.
It remains to make the change of coordinates in the other 
terms involving the first order differentials, which is 
straightforward. 

Note that the radial part $R^\mu(\Om)$ of the Casimir operator 
is symmetric with respect to the inner product 
$\langle \Phi,\Psi\rangle = 
\frac{1}{c}\int_A \Tr\bigl( \Phi(a) (\Psi(a))^\ast\bigr) |\de(a)|\, da$ by \eqref{eq:PhimulaeigenfunctionCasimir}, \eqref{eq:orthorelMatrixSpherF} 
for matrix spherical functions $\Phi$, $\Psi$, and 
the results given in Subsection \ref{ssec:generalsetup}. 
Since the second order matrix partial differential operator is obtained by conjugation by $Q_0$, we obtain the symmetry.
\end{proof}

\section{The leading term of $\Phi^\mu_\la$}
\label{sec:Leadingcoeffgencase}

In Section \ref{sec:leadingcoefQnu} we have introduced 
the leading term $Q^\mu_\nu$ of the matrix spherical
functions for $\Phi^\mu_\nu$ for $\nu\in B(\mu)$. 
Using these results we can determine the 
leading term $Q^\mu_\la$ of the matrix spherical
functions for $\Phi^\mu_\la$ for $\la\in P^+_G(\mu)$. 
We do this by introducing the leading term from 
an embedding of $V^K_\mu$ in a large tensor product representation, similarly to the construction in 
Section \ref{sec:leadingcoefQnu}. We then show by using the radial part of the Casimir operator, that this is indeed a leading term by establishing the lower triangularity of the radial part of the Casimir operator on these functions. 

Assume as before $\mu=a\om_1+b\om_2$ with $a,b\in \N$ and we take
$\la\in P^+_G(\mu)$.  By  Condition \ref{cond:BmustructurePGmu} 
 we can write
$\la= \nu_i+d_1\la_1+d_2\la_2$ with $\nu_i\in B(\mu)$, 
$d_1,d_2\in \N$. Generalising the 
construction of $\psi_1$, $\psi_2$ and $Q^\mu_{\nu_i}$ 
as in Section \ref{sec:specialcases} 
and \ref{sec:leadingcoefQnu}, we define the tensor 
product representation and an explicit element by 
\begin{equation*}
W^G_\la = 
\bigl( V^G_{\om_1}\otimes V^G_{\om_{m+1}}\bigr)^{\otimes d_1}
\otimes
\bigl( V^G_{\om_2}\otimes V^G_{\om_{m}}\bigr)^{\otimes d_2}
\otimes U^G_{\nu_i}, \qquad
w = v_1^{\otimes d_1}\otimes v_2^{\otimes d_2} \otimes u \in W^G_\la
\end{equation*}
using the notation of Lemmas \ref{lem:explicitvKfixedvectorpsi1}
and \ref{lem:explicitvKfixedvectorpsi2}
and \eqref{eq:defUGnuirepr}. 
Using the results of Sections \ref{sec:specialcases} and  \ref{sec:leadingcoefQnu} we see that 
$w$ is a $K$-highest weight vector of highest weight $\mu$ in $W^G_\la$. So we get a $K$-intertwiner 
$j\colon V^K_\mu \to W^G_\la$ mapping the 
highest weight vector $v_\mu\in V^K_\mu$ to $w$. 

\begin{prop}\label{prop:leadingcoeffPhimulambda}
Define the matrix spherical function $Q^\mu_\la\colon G \to \End(V^K_\mu)$ by 
$Q^\mu_\la(g) = j^\ast \circ \pi^G_{W^G_\la}(g) \circ j$, then 
\begin{gather*}
Q^\mu_\la(a_\bt) = \bigl( \psi_1(a_\bt)\bigr)^{d_1}
\bigl( \psi_2(a_\bt)\bigr)^{d_2}\, Q^\mu_{\nu_i}(a_\bt)
\end{gather*}
and $Q^\mu_\la\vert_A = \sum_{\la'\preccurlyeq\la; \la'\in P^+_G(\mu)}
a_{\la'} \Phi^\mu_{\la'}\vert_A$ for constants $a_{\la'}$. 
\end{prop}

Note in particular, that the action of $Q^\mu_\la(a_\bt)$ 
on the one-dimensional 
constituent $V^M_{\si_k}$ in $V^K_\mu$ is given by
\begin{equation}
\bigl( \psi_1(a_\bt)\bigr)^{d_1}
\bigl( \psi_2(a_\bt)\bigr)^{d_2}\, q^\mu_{\nu_i,\si_k}(a_\bt)
\end{equation}
which is a homogeneous polynomial in $(\cos t_1, \cos t_2)$ 
of degree $2d_1+4d_2+2a+4b+2i$, see Remark 
\ref{rmk:prop:matrixentriesatinUnui} and 
Theorem \ref{thm:Qmunuisasmatrixspherfunction}. 

\begin{proof} As noted, $w$ is a highest weight vector for the action of $K$ of highest weight $\mu$, so by construction 
$Q^\mu_\la$ is a matrix spherical function, and by 
Subsection  \ref{ssec:generalsetup} it is a linear 
combination of $\Phi^\mu_\la$ for $\la\in P^+_G(\mu)$ 
by the Peter-Weyl theorem. 
Since we have the decomposition $W^G_\la = \bigoplus_{\la'\preccurlyeq \la} n_{\la'}
V^G_\la$, with $n_\la=1$, by repeated application 
of e.g. \cite[Lemma $\mathrm{(3.1)}$]{Kuma}, the 
expression for $Q^\mu_\la\vert_A$ follows. 

For the proof of the explicit expression, we use the notation as in the proof of Proposition \ref{prop:matrixentriesatinUnui}. 
Since the matrix entry of $a_\bt(r,s)$ acting on $v_i$ and taking inner product with $v_i$ is $\psi_i(a_\bt)$ for $i=1,2$ by 
Lemmas \ref{lem:explicitvKfixedvectorpsi1} and 
\ref{lem:explicitvKfixedvectorpsi2}, we find the 
result from Theorem \ref{thm:Qmunuisasmatrixspherfunction}. 
\end{proof}

In order to understand the decomposition of $Q^\mu_\la$ 
of Proposition \ref{prop:leadingcoeffPhimulambda} we
calculate the action of the radial part of the 
Casimir operator on $Q^\mu_\la$ as a function on $A$. 
Recall 
\eqref{eq:RadpartLeibniz}, and take $f$ a polynomial in $(\psi_1,\psi_2)$, then  
this leads to Proposition \ref{prop:RadpartOmenleadingtermQmulambda}. 
Note that Proposition 
\ref{prop:RadpartOmenleadingtermQmulambda} generalises
Lemma \ref{lem:acrionradpartonQs}, but Lemma 
\ref{lem:acrionradpartonQs} is used in the proof 
of Proposition \ref{prop:RadpartOmenleadingtermQmulambda}. 

\begin{prop}\label{prop:RadpartOmenleadingtermQmulambda}
We have as functions on $A$, 
\begin{gather*}
R^\mu(\Om) Q^\mu_\la = c_\la Q^\mu_\la + 
\sum_{\la'\prec\la; \la'\in P^+_G(\mu)}
b_{\la'} Q^\mu_{\la'}
\end{gather*}
\end{prop}

\begin{corollary}\label{cor:prop:RadpartOmenleadingtermQmulambda}
In Proposition \ref{prop:leadingcoeffPhimulambda}
we have $a_\la\not=0$, so that there exists constants
$b_{\la'}$ with 
\[
\Phi^\mu_\la= \sum_{\la' \preccurlyeq \la}
b_{\la'} Q^\mu_\la, \qquad b_\la\not=0.
\]
\end{corollary}

The statement of Corollary \ref{cor:prop:RadpartOmenleadingtermQmulambda} 
motivates to call the matrix spherical function $Q^\mu_\la$
the leading term of $\Phi^\mu_\la$.

\begin{proof}
In case $a_\la=0$ in Proposition \ref{prop:leadingcoeffPhimulambda}, we have $Q^\mu_\la$ in the span
of $\Phi^\mu_{\la'}$ for $\la'\prec\la$ which is an 
invariant space for the radial part of the Casimir 
$R^\mu(\Om)$ with eigenvalues $c_{\la'}$. 
By Lemma \ref{lem:eigvalueOmdifferent} the 
eigenvalue $c_\la$ is not contained in this set,
but Proposition \ref{prop:RadpartOmenleadingtermQmulambda} 
and Proposition \ref{prop:leadingcoeffPhimulambda} applied
to $\la'\prec\la$ shows that the eigenvalue $c_\la$ has to
occur, since $R^\mu(\Om)$ acts in a lower triangular way
on the $Q^\mu_{\la'}$'s. This is the required contradiction. 

So this means that  
we can invert the relation of 
Proposition \ref{prop:leadingcoeffPhimulambda},
giving the stated expansion.
\end{proof}

\begin{proof}[Proof of Proposition \ref{prop:RadpartOmenleadingtermQmulambda}] Put $f(\psi_1,\psi_2)=\psi_1^{d_1}\psi_2^{d_2}$,
then the first term on the right hand side of 
\eqref{eq:RadpartLeibniz} follows from 
Lemma \ref{lem:acrionradpartonQs}. For the second 
term we have by a calculation 
\begin{gather}
\label{eq:prop:RadpartOmenleadingtermQmulambda1}
R^0(\mu)f = 
2(d_1^2 +d_1(1+2d_2+m)+2d_2^2+2md_2)\psi_1^{d_1}\psi_2^{d_2}
-2d_2^2 \psi_1^{d_1+1}\psi_2^{d_2-1}\\
\nonumber 
\qquad\qquad 
-2d_1(d_1+4d_2+3)\psi_1^{d_1-1}\psi_2^{d_2}
-4d_1(d_1-1)\psi_1^{d_1-2}\psi_2^{d_2+1},
\end{gather}
which follows from the explicit expression of 
the radial part of the Casimir operator, $R^0(\Om)$, 
in the $(\psi_1,\psi_2)$-coordinates, see 
Remark \ref{rmk:lem:PDEforQ}. 
For the final term 
$- \sum_{p=1}^2 \frac{\partial f}{\partial t_p}
\frac{\partial Q^\mu_{\nu_r}}{\partial t_p}$ of 
\eqref{eq:RadpartLeibniz} we  consider the action 
on the constituent $V^M_{\si_k}$ in $V^K_\mu$, and we 
find 
\begin{gather*}
- \sum_{s=1}^2 \frac{\partial f}{\partial \psi_s} 
\bigl( C_s(\psi_,\psi_2)\, Q_0 \bigr)_{i,k}
\end{gather*}
using  the chain rule and Lemma \ref{lem:dQdpsi}. By the 
explicit expression of Lemma \ref{lem:dQdpsi} this term gives 
\begin{gather}
\label{eq:prop:RadpartOmenleadingtermQmulambda2}
-2\bigl(d_1\psi_2 +d_2\psi_1\bigr) 
\bigl( i\psi_1^{d_1-1}\psi_2^{d_2}\, Q^\mu_{\nu_{i-1}} 
+ (a-i) \psi_1^{d_1-1}\psi_2^{d_2-1}\, Q^\mu_{\nu_{i+1}}\bigr) - \\
\nonumber
2 \bigl( (a+b+2i)d_1\psi_2 + (b+2i)d_2\psi_1^2 
- (d_1(a+b+i) +d_2(a+2b+2i))\psi_1\psi_2
\bigr)
\psi_1^{d_1-1}\psi_2^{d_2-1}\, Q^\mu_{\nu_{i}}
\end{gather}
So from \eqref{eq:RadpartLeibniz}, Lemma 
\ref{lem:acrionradpartonQs} and 
\eqref{eq:prop:RadpartOmenleadingtermQmulambda1}, 
\eqref{eq:prop:RadpartOmenleadingtermQmulambda2}
we collect the coefficient of $Q^\mu_\la$ 
in $R^\mu(\Om) Q^\mu_\la$ as
\begin{gather*}
c_{\nu_i} + 2(d_1^2+d_1(1+2d_2+m)+2d_2^2+2md_2) +2(a+b+i)d_1+
2(a+2b+2i)d_2. 
\end{gather*}
Write $\la =\nu_i+\la_\sph$, with $\la_\sph=d_1\la_1+d_2\la_2$, then the eigenvalue $c_\la$ can be written as 
\[
c_\la= c_{\nu_i} + \langle \la_\sph, \la_\sph\rangle 
+ 2\langle \la_\sph, \nu_i + \rho\rangle.
\]
Since 
$\langle \la_\sph, \la_\sph\rangle = 2d_1^2 + 4 d_1d_2+4d_2^2$,
and $\langle \la_1,\nu_i+\rho\rangle = a+b+i+m+1$,
$\langle \la_2,\nu_i+\rho\rangle = a+2b+2i+2m$, we see that 
the coefficient of $Q^\mu_\la$ in $R^\mu(\Om) Q^\mu_\la$ is $c_\la$.
\end{proof}

Note that the proof of 
Proposition \ref{prop:RadpartOmenleadingtermQmulambda} 
actually gives a complete expression for the action 
of $R^\mu(\Om)$ on $Q^\mu_{\la}$. For completeness we
list in Table 
\ref{tab:prop:RadpartOmenleadingtermQmulambda} 
the $\la'\prec\la$ for which $Q^\mu_{\la'}$
occurs with a non-zero $b_{\la'}$ whose explicit value is 
listed as well. 

\begin{center}
\begin{table}[ht]
\begin{tabular}{|l|l|}
\hline 
$\la'$ & $b_{\la'}$ \\
\hline
$(d_1-1)\la_1 + d_2\la_2 +\nu_i$ & $-2d_1(d_1+4d_2+3)-2d_1(a+2b+2i)$ \\ 
$(d_1-2)\la_1 + (d_2+1)\la_2 +\nu_i$ & $-2d_1(d_1-1)$ \\ 
$(d_1+1)\la_1 + (d_2-1)\la_2 +\nu_i$ & $-2d_2^2-2d_2(b+i)$ \\ 
$d_1\la_1 + d_2\la_2 +\nu_{i-1}$ & $-2i(b+i)-2id_2$ \\ 
$(d_1-1)\la_1 + (d_2+1)\la_2 +\nu_{i-1}$ & $-2id_1$ \\ 
$(d_1-1)\la_1 + d_2\la_2 +\nu_{i+1}$ & $-2(a-i)d_1$ \\ 
$d_1\la_1 + (d_2-1)\la_2 +\nu_{i+1}$ & $-2(a-i)d_2$ \\ 
\hline
\end{tabular}
\vspace*{4pt}
\caption{Table for the remaining coefficients in 
Proposition \ref{prop:RadpartOmenleadingtermQmulambda}.}
\label{tab:prop:RadpartOmenleadingtermQmulambda}
\end{table}
\end{center}

Note that indeed all $\la'$ satisfy $\la'\in P^+_G(\mu)$ 
and $\la'\prec\la$, which can be checked using the 
results of Section \ref{sec:structuremultfree}.

\section{The case $\mu=a\om_1+b\om_2$ with $b$ negative}
\label{sec:bnegative}

In general, we obtain from \eqref{eq:dualofPhimula} 
and $\si_k(\mu^\ast)=\si_{a-k}(\mu)$ for $\mu=a\om_1+b\om_2$,
$a\in\N$, $b\in \Z$, see Subsection \ref{ssec:multfreetriples}, and $a_\bt^{-1}=a_{-\bt}$, 
\begin{equation*}
q^{\mu^\ast}_{\nu_r(\mu^\ast), \si_k(\mu^\ast)}(a_\bt) 
=
q^{\mu}_{\nu_{a-r}(\mu), \si_{a-k}(\mu)}(a_{-\bt}) 
\end{equation*}
extending to the notation \eqref{eq:defphimulafromPhimula} 
to more general $\mu$ and stressing the dependence on
$\mu$ and $\mu^\ast$ in the corresponding weights. So 
for the corresponding $\Phi_0^\mu$, we obtain 
\begin{equation}\label{eq:PhimustartPhimu}
\Phi^{\mu^\ast}_0(a_\bt) = 
J \Phi^\mu_0(a_{-\bt})J, \qquad 
J_{i,j}=\de_{i+j,a}, \ \ 0\leq i,j\leq a. 
\end{equation}
Applying \eqref{eq:dualofPhimula} to \eqref{eq:expPhumulaintoPhibottom} using that $\la_\sph^\ast=\la_\sph$ and that spherical functions 
satisfy $\phi(a_\bt)=\phi(a_{-\bt})$ we obtain 
for the matrix polynomials $P^\mu_{\bd}
(\phi_1,\phi_2)= P_{\bd}
(\phi_1,\phi_2)$ introduced in 
\eqref{eq:MatrixOPorthogonality}
\begin{equation}\label{eq:MPolsmustartMPolsmu}
P^{\mu^\ast}_{\bd}(\phi_1,\phi_2)= 
JP^\mu_{\bd} (\phi_1,\phi_2)J.  
\end{equation}
The weight function satisfies 
$W^{\mu^\ast}(\phi_1,\phi_2)=JW^{\mu}(\phi_1,\phi_2)J$
as follows from \eqref{eq:PhimustartPhimu}, so that we 
see that the matrix polynomials for $\mu=a\om_1+b\om_2$
and $\mu^\ast = a\om_1 -(a+b)\om_2$ are essentially the 
same. So this covers the case $b\leq -a$. 

It remains to consider the case $-a<b<0$ with $a\in \N$, $b\in \Z$, and using duality we can restrict to the case $-\frac12 a \leq b<0$. However, in this case we cannot extend the 
method established for the case $b\geq 0$ easily, due to the fact that the bottom splits into two parts. The results for each of these parts cannot be easily related to each other.

\begin{remark}\label{rmk:selfdualreducible} The case 
that 
$\mu^\ast=\mu$, i.e. $a\in 2\N$ and $b=-\frac12a$ or $\mu=2c\om_1-c\om_2$ for $c\in \N$, is exhibiting different behaviour. Assume
$c\geq 1$, we see that the corresponding space
$A$ and $\cA$ as in Proposition \ref{prop:Sisindecomposable} 
for the matrix weight $W$, see 
Remark \ref{rmk:prop:Sisindecomposable}, are no longer trivial, since $\cA'$ and $\cA$ both contain $J$. 
Calculations for small values of $c$ in $\mu=2c\om_1-c\om_2$  indicate that we may expect $\cA'=\C J\oplus \C\Id$ and $\cA=\R J \oplus \R \Id$ with $\cA'$ and $\cA$ defined as in 
Proposition \ref{prop:Sisindecomposable}. 
\end{remark}

Note that in the study of matrix orthogonal polynomials of a single variable 
related to $(\SU(2)\times \SU(2), \text{diag})$ 
the weight is also reducible, see \cite[Prop.~6.4, Thm.~6.5]{KoelvPR-IMRN}. In that case the algebra $\cA'$ is also two-dimensional with a similarly defined non-trivial element. 
So we see that self-duality of the $K$-representations in these cases leads to reducibility of the weight for the corresponding matrix orthogonal polynomials. The precise
relation requires more attention in general.


\appendix


\section{The radial part of the Casimir operator}\label{sec:AppRadial}

In general the determination of the radial part of an 
operator arising from a suitable element in the 
universal enveloping algebra is due to Harish-Chandra in unpublished papers from 1960, see \cite{Hari}.
The result is mainly used for representations of noncompact 
Lie groups, see \cite[Ch.~VIII]{Knap-Overview}, 
\cite[Ch.~9]{Warn-2}. In this case we need to do this for 
the compact setting, and we derive the explicit expression
from the Casimir element in the centre of $U(\Lg)$. 
For this we follow Casselman and Mili\v{c}i\'{c} \cite{CassM}. 

\subsection{Structure theory}\label{ssec:appAstructurethy}

In order to calculate the radial part of the Casimir operator
following \cite{CassM}, we note that $K= G^\theta$ with 
$\theta(g) = JgJ$, $J=\diag(-1,-1,1,\cdots,1)$. In order to do the 
calculation we conjugate to the maximally split case. So we take 
\begin{equation}\label{eq:defJ}
J' = \begin{pmatrix} 0 & 0 & J_{2} \\
                      0 & I_{m-2} & 0 \\
                      J_{2} & 0 & 0\\ \end{pmatrix}, \quad
u = \begin{pmatrix} \frac{1}{\sqrt{2}}I_2 & 0 & \frac{1}{\sqrt{2}}J_{2} \\
                      0 & I_{m-2} & 0 \\
                      -\frac{1}{\sqrt{2}}J_{2} & 0 & \frac{1}{\sqrt{2}}I_2\\ \end{pmatrix} \in \SU(m+2), \quad
u^\ast J'u = J,
\end{equation}
where $J_2=\begin{pmatrix}0 & 1\\1& 0\end{pmatrix}$ 
and $\theta'(g)=J'gJ'$, so that $u\theta(g)u^\ast = \theta'(u^\ast g u)$  and $K'= G^{\theta'}=uKu^\ast$. We use the same notation for the involutions $\theta$ and $\theta'$ for the complexified Lie algebras. 
Now $\Lg=\mathfrak{sl}(m+2,\C)$ has the root system $\De =
\{\ep_i-\ep_j\}_{1\leq i\not=j\leq m+2}$,
$\Lg = \Lh \oplus \bigoplus_{\al\in \De} \Lg_\al$, where
$\Lh$ is the Cartan subalgebra consisting of the diagonal
elements in $\Lg$. 
The matrix 
$E_{i,j}$  spans
$\Lg_{\ep_i-\ep_j}$, see Subsection \ref{ssec:structurethy}. 
Then $\La'=u\La u^\ast$ consists of diagonal matrices 
$X=\diag(d_1,d_2,0\cdots,0, -d_2,-d_1)$, and we let 
$f_i(X)=d_i$, $i=1,2$. Then the reduced root system $R$ is of type
$\mathrm{BC}_2$ 
and the identification is given
in Figure \ref{fig:rootsystem}. Then the positive roots of $\De$ and $R$ correspond to each other. 
Moreover, $\Lm'=u\Lm u^\ast = \Lm$.
With $A'=uA u^\ast$, and $a'_\bt = 
ua_\bt u^\ast=  \diag (e^{it_1}, e^{it_2}, 1\cdots, 1, e^{-it_2}, e^{-it_1})$ we have $M'= Z_{K'}(A') = u Z_{K}(A) u^\ast = M$. 
Let $n_1f_1+n_2f_2$ be the character of $A$ sending 
$a'_\bt\mapsto e^{i(n_1t_1+n_2t_2)}$. 

\begin{figure}\label{fig:rootsystem}
\begin{tabular}{|l|c|l|}
\hline
$\be\in R$ & $\dim \Lg_\be$ & $\al\in \De$ with $\al\vert_{\La'}=\be$ \\
\hline
$f_1-f_2$ & $2$ & $\ep_1-\ep_2$, $\ep_{m+1}-\ep_{m+2}$  \\
$f_1+f_2$ & $2$ & $\ep_{i}-\ep_{m+i}$, $i=1,2$\\
$2f_i$, $1\leq i\leq 2$ & $1$  & $\ep_{i}-\ep_{m+3-i}$ \\
$f_i$, $1\leq i\leq 2$ & $2(m-2)$  & $\ep_{i}-\ep_{2+j}$, $\ep_{2+j}-\ep_{m+3-i}$, $1\leq j\leq m-2$ \\
$f_2-f_1$ & $2$ & $\ep_2-\ep_1$, $\ep_{m+2}-\ep_{m+1}$  \\
$-f_1-f_2$ &$2$ & $\ep_{m+i}-\ep_{i}$, $i=1,2$ \\
$-2f_i$, $1\leq i\leq 2$ & $1$  & $\ep_{m+3-i}-\ep_{i}$ \\
$-f_i$, $1\leq i\leq 2$ & $2(m-2)$  & $\ep_{2+j}-\ep_{i}$, $\ep_{m+3-i}-\ep_{2+j}$, $1\leq j\leq m-2$ \\
\hline
\end{tabular}
\caption{The restricted root system of type $\mathrm{BC}_2$.}
\end{figure}

Then the root space decomposition for the action of $A'$ is given
by 
\[
\Lg = \La' \oplus \Lm' \oplus \bigoplus_{\be\in R} \Lg_\be, 
\qquad \Lg_\be = \bigoplus_{\al\in \De, \al\vert_{\La'} = \be} \Lg_\al
\]
where, for $\al=\ep_i-\ep_j\in \De$, $\Lg_\al=\C Y_\al$ with $Y_\al=E_{i,j}$, where we use the same notation $\be$ for the corresponding derivative $\be\colon \La'\to \C$. 
Note that $\theta'$ gives an action on $\De$ by 
$\theta'(\al)(H)=\al(\theta'(H))$ for $H\in \Lh$. Then 
$-\theta'$ is an involution of $\{\al\in \De \mid 
\al\vert_{\La'} = \be\}$ for $\be\in R$. 

\subsection{Casimir element}\label{ssec:appACasimirelt}

The Killing form on $\Lg$ is given by $B(X,Y)=\Tr(XY)$ up to 
a positive multiple, and 
the Casimir element $\Om= \sum_{i} X_iX_i^\ast \in Z(U(\Lg))$
where $\{X_i\}_i$ is a basis for $\Lg$ and $\{X_i^\ast\}_i$ 
its dual basis with respect to $B$. 
Put $H_i=E_{i,i}-E_{m+3-i,m+3-i}$, $i=1,2$, as the basis for 
$\La'$, then $H_i^\ast = \frac12 H_i$ and note that $E_{i,j}^\ast=E_{j,i}$ for $i\not=j$, or
$Y_\al^\ast=Y_{-\al}$. Observe that $B\vert_{\Lm\times \Lm}$ is 
non-degenerate, and let $\Om_{\Lm}$ be the corresponding
Casimir element. So we get 
\begin{equation}\label{eq:AppCasimir1}
\Om = \Om_\Lm + \frac12 \sum_{i=1}^2 H_i^2 
+ \sum_{\be\in R^+} \sum_{\stackrel{\scriptstyle{\al\in \De^+}}{\scriptstyle{\al\vert_{\La'} = \be}}}
\bigl( Y_\al Y_{-\al}+Y_{-\al}Y_{\al}\bigr).
\end{equation}
Now we want to rewrite \eqref{eq:AppCasimir1} following 
\cite[\S 2]{CassM}. So let $a\in A'_\reg$, i.e. $\be(a)\not=\pm 1$
for all $\be\in R^+$. Define $X^a=\Ad(a^{-1})X$, $X\in U(\Lg)$, and
let $\al\in \De$ with $\al\vert_{\La'} = \be$. Then, see 
\cite[Lemma~2.2]{CassM}, 
\begin{equation}\label{eq:AppCasimir2}
X_\al = Y_\al+\theta' Y_\al =
Y_\al+Y_{\theta'\al}\in \Lk', \qquad 
Y_\al = \frac{\be(a)}{1-\be(a)^2}(X_\al^a -\be(a)X_\al).
\end{equation}
In order to obtain the infinitesimal Cartan decomposition of the Casimir element $\Ga_a^{-1}(\Om)$, see \cite[Theorem~2.1]{CassM}, we need to write $\Om$ as a sum of elements of the form 
$X^aHY$ with $X,Y\in U(\Lk')$, $H\in U(\La')$. Note that the 
first two terms in \eqref{eq:AppCasimir1} are of the right form. 
Using \eqref{eq:AppCasimir2} we see that 
\begin{gather*}
\sum_{\stackrel{\scriptstyle{\al\in \De^+}}{\scriptstyle{\al\vert_{\La'} = \be}}}
\bigl( Y_\al Y_{-\al}+Y_{-\al}Y_{\al}\bigr)= 
\frac{-1}{(\be(a)-\be(a)^{-1})^2}
\sum_{\stackrel{\scriptstyle{\al\in \De^+}}{\scriptstyle{\al\vert_{\La'} = \be}}}
\bigl( X_\al^aX_{-\al}^a +X_{-\al}^aX_{\al}^a 
+  X_\al X_{-\al} +X_{-\al}X_{\al} \\
-\be(a)^{-1} X_\al^a X_{-\al} - \be(a) X^a_{-\al}X_\al
- \be(a) X_\al X^a_{-\al} - \be(a)^{-1} X_{-\al}X_\al^a\bigr).
\end{gather*}
Next observe 
$\sum_{\al\in \De^+;\al\vert_{\La'} = \be} X_{-\al}^aX_{\al}^a
= \sum_{\al\in \De^+;\al\vert_{\La'} = \be} X_{\theta'\al}^a
X_{-\theta'\al}^a = \sum_{\al\in \De^+;\al\vert_{\La'} = \be} X_{\al}^aX_{-\al}^a$ using the involution $-\theta'$ and 
$X_\al = X_{\theta'\al}$. Similarly, we can take other terms together. Then only the last two terms are not yet of the right form. 

\begin{lemma}\label{lem:AppAcommrelXaX} For $\al\in \De^+$ 
with $\al\vert_{\La'} = \be$ we have 
\[
[X_\al^a, X_{-\al}]+ [X^a_{-\theta'\al}, X_{\theta'\al}]
= (\be(a)^{-1}-\be(a)) (H_\al+H_{-\theta'\al}) \in \La'
\]
where $H_{e_i-e_j}= E_{i,i}-E_{j,j}$.
\end{lemma}

\begin{proof} Using \eqref{eq:AppCasimir2} we rewrite 
the commutators in terms of the $Y_\al$'s. The mixed terms cancel
and we are left with 
\[
[X_\al^a, X_{-\al}]+ [X^a_{-\theta'\al}, X_{\theta'\al}] =
(\be(a)^{-1}-\be(a))[Y_\al, Y_{-\al}]  +
(\be(a)-\be(a)^{-1}) [Y_{\theta'\al}, Y_{-\theta'\al}]
\]
in terms of commutators of the $Y_\al$'s. 
Since the right hand side is in $\Lh$ and in the $-1$-eigenspace 
of $\theta'$ we see that is contained in $\La'$. 
\end{proof}

Using this in the expression for the Casimir element leads to 
the infinitesimal Cartan decomposition for $\Om$:
\begin{gather}\label{eq:AppCasimir3}
\Om = \Om_\Lm + \frac12 \sum_{i=1}^2 H_i^2 
+ \frac12 \sum_{\be\in R^+} 
\frac{\be(a)+\be(a)^{-1}}{\be(a)-\be(a)^{-1}} \dim\Lg_\be\, H_\be 
+ \\ 
2\sum_{\be\in R^+} 
\frac{\be(a)+\be(a)^{-1}}{(\be(a)-\be(a)^{-1})^2}
\sum_{\stackrel{\scriptstyle{\al\in \De^+}}{\scriptstyle{\al\vert_{\La'} = \be}}} X^a_{\al}X_{-\al} 
- 2 
\sum_{\be\in R^+} 
\frac{1}{(\be(a)-\be(a)^{-1})^2}
\sum_{\stackrel{\scriptstyle{\al\in \De^+}}{\scriptstyle{\al\vert_{\La'} = \be}}}
X^a_{\al}X^a_{-\al} + X_{\al}X_{-\al}, \nonumber
\end{gather}
where $H_\be=n_1H_1+n_2H_2$ for $\be=n_1f_1+n_2f_2$. 

\subsection{The left invariant differential operator 
corresponding to the Casimir element}\label{ssec:appAleftoinvDOCasimirelt}
Let $F\colon G\to \End(V^{K'}_\mu)$, where $V^{K'}_\mu$
is the same representation space as $V^K_\mu$, and the action
is given by $\pi^{K'}_\mu(k')=\pi^K_\mu(u^{\ast}k'u)$, $k'\in K'$.
We assume $F$ satisfies $F(k_1'gk_2')=\pi^{K'}_\mu(k_1')
F(g) \pi^{K'}_\mu(k_2')$, so that $F$ is determined by its 
restriction to $A'$ and, since $M'=M$, we have $F\colon A'\to \End_M(V^{K'}_\mu)$. Now the action of $\Om$ as a left invariant
operator satisfies $\bigl( \Om\cdot F\bigr)\vert_{A'} =
R(\Om)\cdot (F\vert_{A'})$, where $R(\Om)$ is the 
radial part of the Casimir element. 
In the decomposition \eqref{eq:AppCasimir3}, $\Om_\Lm$
acts as a scalar on each $M$-type by Schur's Lemma. 
So the action of $\Om_\Lm$ on $F\vert_{A'}$ is 
by multiplying by a diagonal constant matrix. The second
term acts as a second order differential operator, and the 
third term as a first order differential operator by 
observing that, after putting $f(t_1,t_2)= F(a'_\bt)$ we
have $iH_p\cdot f = \frac{\partial f}{\partial t_p}$. 
The action of 
the differential operators do not involve the $M$-type. Then 
$X^a_{\al}X_{-\al}\cdot (F\vert_{A'})=
\pi^{K'}_\mu(X_\al) (F\vert_{A'})\pi^{K'}_\mu(X_{-\al})$, and 
similarly 
$X^a_{\al}X^a_{-\al}\cdot (F\vert_{A'})=
\pi^{K'}_\mu(X_\al X_{-\al}) (F\vert_{A'})$ 
and 
$X_{\al}X_{-\al}\cdot (F\vert_{A'})=
(F\vert_{A'})\pi^{K'}_\mu(X_\al X_{-\al})$, see \cite{CassM}, 
where we use the same notation for the representation of 
the Lie algebra. In order to calculate these terms, we 
restrict to the $K$-representation of highest weight 
$\mu=a\om_1+b\om_2$, $a\in \N$, $b\in \Z$. We can then 
read off $X_\al$ using Figure \ref{fig:rootsystem}, and 
next see which entry of  $u^\ast X_\al u$ is in the 
upper left $2\times 2$-block. Finally, we conjugate back 
and we find the following expression for the radial 
part of the Casimir operator for 
a function $F\colon A \to \End_M(V^K_\mu)$ for 
$\mu=a\om_1+b\om_2$, $a\in \N$, $b\in \Z$, where 
$G(t_1,t_2)=F(a_\bt)$:
\begin{gather}\label{eq:AppCasimir4}
\bigl( R^\mu(\Om)G\bigr) (t_1,t_2) = 
R^\mu(\Om_\Lm) G(t_1,t_2) - \frac12 \sum_{p=1}^2 \frac{\partial^2G}{\partial t_p^2}(t_1,t_2) +  \\ \nonumber
\bigl(R^\mu_s(\Om)G\bigr) (t_1,t_2) + \bigl(R^\mu_m(\Om)G \bigr)(t_1,t_2) +\bigl(R^\mu_l(\Om)G\bigr) (t_1,t_2), 
\end{gather}
where the action is split according to the short, middle and long roots of $\mathrm{BC}_2$. We obtain 
\begin{gather*}
\bigl(R^\mu_s(\Om)G\bigr) (t_1,t_2) = -(m-2) 
\sum_{i=1}^2 \frac{\cos t_i}{\sin t_i} \frac{\partial G}{\partial t_i}(t_1,t_2)
\end{gather*} 
since for the short roots $f_i$ the element $u^\ast X_\al u$ 
is not contained in the upper left $2\times 2$-block, and so the 
last three terms in \eqref{eq:AppCasimir3} do not contribute for the short roots. So the operator $R^\mu_s(\Om)$ is 
independent of the $K$-representation $\pi^K_\mu$. 
For the middle roots $f_1\pm f_2$ we get 
that the operator $R^\mu_m(\Om)$ is defined by
$\bigl(R^\mu_m(\Om)G\bigr) (t_1,t_2) =$
\begin{gather*}
- \frac{\cos(t_1+t_2)}{\sin(t_1+t_2)} \Bigl(
\frac{\partial G}{\partial t_1}(t_1,t_2)+
\frac{\partial G}{\partial t_2}(t_1,t_2)\Bigr) 
- \frac{\cos(t_1-t_2)}{\sin(t_1-t_2)} \Bigl(
\frac{\partial G}{\partial t_1}(t_1,t_2)-
\frac{\partial G}{\partial t_2}(t_1,t_2)\Bigr) \\
- \Bigl(\frac{\cos(t_1+t_2)}{\sin^2(t_1+t_2)}
+ \frac{\cos(t_1-t_2)}{\sin^2(t_1-t_2)}\Bigr)
\bigl( \pi^K_\mu(E_1) G(t_1,t_2) \pi^K_\mu(F_1) 
+ \pi^K_\mu(F_1) G(t_1,t_2) \pi^K_\mu(E_1)  \bigr) + \\
\frac12
\Bigl( \frac{1}{\sin^2(t_1+t_2)} + 
\frac{1}{\sin^2(t_1-t_2)}\Bigr)
\bigl( \pi^K_\mu(E_1F_1+F_1E_1) G(t_1,t_2)  
+ G(t_1,t_2) \pi^K_\mu(E_1F_1+F_1E_1)  \bigr) 
\end{gather*} 
and for the long roots $2f_i$ we get 
\begin{gather*}
\bigl(R^\mu_l(\Om)G\bigr) (t_1,t_2) = - \sum_{i=1}^2  
\frac{\cos(2t_i)}{\sin(2t_i)} 
\frac{\partial G}{\partial t_i}(t_1,t_2)  
- \sum_{i=1}^2 \frac{\cos(2t_i)}{\sin^2(2t_i)}
\pi^K_\mu(E_{i,i}) G(t_1,t_2) \pi^K_\mu(E_{i,i})\\
+ \frac12 \sum_{i=1}^2  
\frac{1}{\sin^2(2t_i)} 
\bigl( \pi^K_\mu(E_{i,i})^2 G(t_1,t_2)  
+ G(t_1,t_2) \pi^K_\mu(E_{i,i})^2  \bigr).
\end{gather*} 
In order to describe the action of $\Om_\Lm$ we only need 
the action on the $1$-dimensional $M$-representation 
$V^M_{\si_k}$ occurring in $V^K_\mu$, see \eqref{eq:decompVKmurestrtoM}. Let 
$M_1= E_{1,1}+E_{m+2,m+2}-\frac{2}{m-2}\sum_{r=3}^m E_{r,r}$
and $M_2= E_{2,2}+E_{m+1,m+1}-\frac{2}{m-2}\sum_{r=3}^m E_{r,r}$,
then the $M_i$'s are orthogonal to the $(m-2)\times (m-2)$-block
of $M$, so that we only need to take the action of 
$M_1$ and $M_2$ into account. Note that $M_1$, respectively
$M_2$, acts as 
$a+b-k$, respectively $b+k$ on $V^M_{\si_k}$. Since
$M_1^\ast=\frac{m}{2(m+2)} M_1-\frac{1}{m+2} M_2$ and 
$M_1^\ast=\frac{m}{2(m+2)} M_2-\frac{1}{m+2} M_1$, 
this gives
\[
R^\mu(\Om_\Lm)\vert_{V^M_{\si_k}}= \frac{1}{2(m+2)} 
\bigl( m (a+b-k)^2 - 4(a+b-k)(b+k) + m (b+k)^2\bigr),
\]
for $k\in \{0,\ldots, a\}$.


\end{document}